\documentclass{amsart}
\usepackage[dvipsnames,svgnames,x11names,hyperref]{xcolor}
\usepackage{url,graphicx,verbatim,amssymb,enumerate,stmaryrd}
\usepackage[pagebackref,colorlinks,citecolor=blue,linkcolor=blue,urlcolor=blue,filecolor=blue]{hyperref}
\usepackage{tikz, tikz-cd} 
\usepackage{enumitem}
\usetikzlibrary{matrix} 
\usetikzlibrary{arrows,decorations.pathmorphing,backgrounds,fit,positioning,shapes.symbols,chains,calc}
\usepackage[all]{xy}
\usepackage[margin=3cm,vmargin=3cm]{geometry}
\usepackage{dsfont}

\newtheorem{theorem}{Theorem}[section]
\newtheorem{lemma}[theorem]{Lemma}
\newtheorem{proposition}[theorem]{Proposition}
\newtheorem{corollary}[theorem]{Corollary}

\theoremstyle{definition}
\newtheorem{definition}[theorem]{Definition}

\newtheorem{remark}[theorem]{Remark}

\newcommand{\cA}{\mathcal{A}}

\newcommand{\rD}{\mathrm{D}}

\DeclareMathOperator{\colim}{colim}

\newcommand{\cI}{\mathcal{I}}

\newcommand{\cM}{\mathcal{M}}
\newcommand{\cO}{\mathcal{O}}

\newcommand{\bt}{\mathbf{t}}

\newcommand{\F}{\mathbb{F}}
\newcommand{\Ch}{\mathrm{Ch}}
\newcommand{\GL}{\mathrm{GL}}
\newcommand{\SL}{\mathrm{SL}}

\newcommand{\im}{\mathrm{im}}
\newcommand{\id}{\mathrm{id}}

\newcommand{\Cone}{\mathrm{Cone}}

\newcommand{\Aut}{\mathrm{Aut}}
\newcommand{\rR}{\mathrm{R}}
\newcommand{\rL}{\mathrm{L}}
\newcommand{\rH}{\mathrm{H}}

\newtheorem{atheorem}{Theorem}

\newtheorem{aapplication}{Application}

\newtheorem*{theoremCrestated}{Theorem~\ref{athm:CGammaSPB}, restated}
\newtheorem*{theoremDduplicate}{Theorem~\ref{athm:connectivity}}

\newcommand{\fm}{\mathfrak{m}}

\newcommand{\PConf}{\mathrm{PConf}}
\newcommand{\lw}{{\textstyle \bigwedge}}
\newcommand{\into}{\hookrightarrow}
\newcommand{\onto}{\twoheadrightarrow}
\newcommand{\FIMod}{\Mod_{\FI}}

\DeclareMathOperator\FBMod{\Mod_{\FB}}

\newcommand{\op}{\text{op}}
\newcommand{\abs}[1]{\left\lvert #1\right\rvert}
\newcommand{\bwedge}{\lw}
\newcommand{\orient}{\lw}
\newcommand{\iso}{\cong}

\newcommand{\bk}{\mathbf{k}}

\newcommand{\FB}{\mathbf{FB}}

\newcommand{\R}{\mathbb{R}}
\newcommand{\Q}{\mathbb{Q}}
\newcommand{\Z}{\mathbb{Z}}
\newcommand{\m}{\to}

\newcommand{\coloneq}{\mathrel{\mathop:}\mkern-1.2mu=}
\newcommand{\arxiv}[1]{\href{http://arxiv.org/abs/#1}{{\tt arXiv:#1}}}

\newcommand{\mr}[1]{{\rm #1}}

\newcommand{\coker}{\mr{coker}}
\newcommand{\Hom}{\mathrm{Hom}}
\newcommand{\FI}{\mathbf{FI}}
\newcommand{\Set}{\mathbf{Set}}
\newcommand{\Sharp}{\mathbf{FI}\sharp}

\newcommand{\Mod}{\mathrm{Mod}}

\newcommand{\Ind}{\mathrm{Ind}}

\newcommand{\FItwist}{\FI^{\rightsquigarrow}}

\newcommand{\FIarrow}{\FI^{\downarrow}}
\newcommand{\FIarrowMod}{\Mod_{\FIarrow}}

\newcommand{\disjoint}{\sqcup}
\newcommand{\proj}[1]{\overline{#1}}
\newcommand{\innerpull}[1]{\overline{#1}}
\newcommand{\outerpull}[1]{\widetilde{#1}}

\newcommand{\HH}{\mathbf{H}^{\FI}}

\newcommand{\iterDelta}{Q}

\DeclareMathOperator{\SPB}{SPB}
\DeclareMathOperator{\SU}{SU}
\DeclareMathOperator{\SUtilde}{\widetilde{\SU}}

\DeclareMathOperator{\stablerange}{SR}
\newcommand{\SR}{\stablerange}

\newcommand{\TypeA}{Type A}
\newcommand{\TypeB}{Type B}

\title{Linear and quadratic ranges in representation stability}
\date{\today}

\author{Thomas Church}
\address{Department of Mathematics, Stanford University, Stanford, CA}
\email{\href{mailto:tfchurch@stanford.edu}{tfchurch@stanford.edu}}
\urladdr{\url{http://math.stanford.edu/~church/}}
\thanks{Thomas Church was supported in part by NSF grant DMS-1350138 and the Alfred P.\ Sloan Foundation FG-2016-6419. He is grateful to the Institute for Advanced Study for their hospitality during the writing of this paper, and to the Friends of the Institute for their support.}

\author{Jeremy Miller}
\address{Department of Mathematics, Purdue University, West Lafayette, IN}
\email{\href{mailto:jeremykmiller@purdue.edu}{jeremykmiller@purdue.edu}}
\urladdr{\url{https://www.math.purdue.edu/~mill1834/}}
\thanks{Jeremy Miller was supported in part by NSF grant DMS-1709726.}

\author{Rohit Nagpal}
\address{Department of Mathematics, University of Chicago, Chicago, IL}
\email{\href{mailto:nagpal@math.uchicago.edu}{nagpal@math.uchicago.edu}}
\urladdr{\url{http://math.uchicago.edu/~nagpal/}}

\author{Jens Reinhold}
\address{Department of Mathematics, Stanford University, Stanford, CA}
\email{\href{mailto:jreinh@stanford.edu}{jreinh@stanford.edu}}
\urladdr{\url{https://web.stanford.edu/~jreinh/}}
\thanks{Jens Reinhold is supported by the E.\ K.\ Potter Stanford Graduate Fellowship}

\begin{document}

\begin{abstract}  
We prove two general results concerning spectral sequences of $\FI$-modules. These results can be used to significantly improve stable ranges in a large portion of the stability theorems for $\FI$-modules currently in the literature. We work this out in detail for the cohomology of configuration spaces where we prove a linear stable range and the homology of congruence subgroups of general linear groups where we prove a quadratic stable range. Previously, the best stable ranges known in these examples were exponential. Up to an additive constant, our work on congruence subgroups verifies a conjecture of Djament.
\end{abstract}

\maketitle
\tableofcontents

\date{\today}

\maketitle

\section{Introduction}
\label{sec:intro}

$\FI$-modules are a convenient framework for studying stability properties of sequences of symmetric group representations.  An $\FI$-module is a functor from the category of finite sets and injections to the category of $\Z$-modules. In this paper, we introduce two new techniques for proving stability  results for graded sequences of $\FI$-modules which yield improved stable ranges in many examples, including the cohomology of configuration spaces and the homology of congruence subgroups of general linear groups. In all applications, this grading will come from the standard grading on the (co)homology of a sequence of spaces or groups. By stability, we roughly mean a bound on the presentation degree in terms of the (co)homological degree. If there is such a bound which is linear in the (co)homological degree, we say that the sequence exhibits a linear stable range (similarly quadratic, exponential etc).

While previous stability arguments focused on bounding the presentation degree, our proof strategy involves studying two other invariants of an $\FI$-module. We call these invariants \emph{stable degree} and \emph{local degree} and show that these invariants are easier to control in spectral sequences than presentation degree. An $\FI$-module over a ring $\bk$ is a functor from $\FI$ to the category of abelian groups that factors thought the category of $\bk$-modules. Finitely generated $\FI$-modules over fields have dimensions that are eventually equal to a polynomial. The stable degree of a finitely generated $\FI$-module is equal to the degree of this polynomial and the local degree controls when these dimensions become equal to this polynomial.

Together, the stable degree and the local degree control the presentation degree of an $\FI$-module (Proposition~\ref{prop:tot-degree1}). Conversely, the presentation degree can be used to bound these invariants (Proposition~\ref{prop:delta} and Theorem~\ref{thm:local-cohomology}). As a consequence, it is enough to bound the stable and the local degrees to bound the presentation degree and vice-versa. More precisely, we have the following quantitative result (Proposition~\ref{prop:tot-degree1}):

\begin{itemize}
\item[($\star$)]  Let $M$ be an $\FI$-module with stable degree $a$ and local degree $b$. Then the generation degree of $M$ is $\le a+b+1$ and the presentation degree of $M$ is~$\le a+2b+2$. 
\item[($\star \star$)]  Let $M$ be an $\FI$-module with generation degree $a$ and presentation degree $b$. Then the stable degree of $M$ is $\le a$ and the local degree of $M$ is $\le a + b -1$. 
\end{itemize} 
 
The two techniques that we introduce for proving stability for sequences of $\FI$-modules involve two different kinds of spectral sequence arguments which we will call \TypeA{} and \TypeB{}. The main results of this paper are two general theorems, one which establishes linear ranges for \TypeA{} arguments and the other which establishes quadratic ranges for \TypeB{} arguments. We will use these general theorems to prove our linear ranges for cohomology of configuration spaces and quadratic ranges for homology of congruence subgroups. 

We note that stable degree and local degree have respectively been previously studied in \cite{ramos} and \cite{djament-vespa} under different names.

\subsection*{\TypeA{} stability arguments}

In \TypeA{} arguments, one constructs a spectral sequence $$E_r^{p,q} \Longrightarrow M^{p+q}$$ where $M^k$ are the objects of interest and with $E^{p,q}_d$  exhibiting stability for some page $d$. One shows that stability is preserved by the spectral sequence to deduce that the $M^k$ stabilize. This strategy was first used in the context of representation stability by Church who proved representation stability for the rational cohomology of ordered configuration spaces (\cite[Theorem~1]{Ch}). It has also been used to establish homological stability results; see e.g. the work of Kupers--Miller--Tran \cite{KMT}. We prove the following theorem which allows one to establish linear stable ranges for sequences of $\FI$-modules using \TypeA{} stability arguments.

\begin{atheorem} 
\label{theorem:non-quillen}
Let $E^{p,q}_r$ be a cohomologically graded first quadrant spectral sequence of $\FI$-modules  converging to $M^{p+q}$. Suppose that for some $d$, the stable and the local degrees of $E^{p,q}_d$ are bounded linearly in $p$ and $q$. Then the same holds for $E^{p,q}_{\infty}$ and $M^{p+q}$.
\end{atheorem}

We use a quantitative version (Proposition~\ref{prop:non-quillen}) of Theorem~\ref{theorem:non-quillen} to establish a linear stable range for the cohomology of configuration spaces with coefficients in an arbitrary abelian group; we suppress this abelian group from the notation whenever convenient. Given a manifold $\cM$, let $\PConf(\cM)$ denote the $\FI^{\op}$-space sending a set $S$ to the space of embeddings of $S$ into $\cM$ (the $S$-labeled configuration space of $\cM$). Taking cohomology gives an $\FI$-module $\rH^k(\PConf(\cM))$. 


\enlargethispage{\baselineskip}
\begin{aapplication} \label{ap:config}
Suppose $\cM$ is a connected manifold of dimension at least $2$. Then we have: \begin{enumerate}
\item The stable degree of $\rH^k(\PConf(\cM))$ is  $ \le 2k$.
\item The local degree of $\rH^k(\PConf(\cM))$ is $ \le \max(-1, 8 k - 2)$.
\item The generation degree of $\rH^k(\PConf(\cM))$ is  $\le \max(0, 10 k - 1)$.
\item The presentation degree of $\rH^k(\PConf(\cM))$ is $ \le \max(0, 18 k - 2)$.
\end{enumerate} 
The same bounds hold for $\rH^k(\PConf(\cM);\bk)$ with any coefficients. In particular, we have: \begin{enumerate}[label=(\alph*)]
\item  If $\cM$ is finite type and $\bk$ is a field, then there are polynomials $p^{\cM}_{k,\bk}$ of degree at most $2k$  such that $$\dim_{\bk}  \rH^k(\PConf_n(\cM);\bk) = p^{\cM}_{k,\bk}(n)$$ if $n >  \max(-1, 8 k - 2)$.
\item The natural map $$\Ind_{S_{n-1}}^{S_n} \rH^k(\PConf_{n-1}(\cM)) \to  \rH^k(\PConf_n(\cM))$$ is surjective for $n > \max(0, 10 k - 1)$, and the kernel of this map is the image of the difference of the two natural maps $$\Ind_{S_{n-2}}^{S_n} \rH^k(\PConf_{n-2}(\cM)) \rightrightarrows \Ind_{S_{n-1}}^{S_n} \rH^k(\PConf_{n-1}(\cM))$$ for $n > \max(0, 18 k - 2)$.
\end{enumerate} 
\end{aapplication}


We establish even better ranges when $\cM$ is at least $3$-dimensional, orientable, or admits a pair of linearly independent vector fields.  The best previously known bounds away from characteristic zero are due to Miller and Wilson who showed that $\dim_{\bk} \rH^k(\PConf_n(\cM); \bk) $ agrees with a polynomial of degree at most $21(k + 1)(1 + \sqrt 2)^{k-2}$ for $n \geq 49(k + 1)(1 + \sqrt 2)^{k-2}$ \cite[Theorem A.12]{MW}. This was proven using the regularity theorem of  Church and Ellenberg \cite[Theorem A]{castelnuovo-regularity}.

The quantitative version (Proposition~\ref{prop:non-quillen}) of Theorem \ref{theorem:non-quillen} can be used to improve bounds for many other sequences of $\FI$-modules. Such examples include homology groups of the generalized configuration spaces of Petersen \cite{Pe}, the  homology of groups of the singular configuration spaces of Tosteson \cite{Tos}, and homotopy groups  of configuration spaces \cite{KM}.

\subsection*{\TypeB{} stability arguments}

In \TypeB{} arguments, one constructs a spectral sequence $E^r_{p,q}$ where $E^1_{0,q}$ are the objects of interest  and one uses highly acyclic simplicial complexes to prove that the spectral sequence converges to $0$ in a range. One then interprets cancellation in this spectral sequence as stability. This method was introduced by Quillen who, in unpublished work, proved homological stability for certain general linear groups. This was first used in the context of representation stability by Putman \cite{Pu} who proved representation stability for the homology of congruence subgroups. 


To state our result for \TypeB{} stability arguments, we will need more terminology. Let $\rH^{\FI}_0(V)$ denote the so-called \emph{minimal generators} of an $\FI$--module. Concretely, $\rH_0^{\FI}(M)_n$ is the cokernel of \[\Ind_{S_{n-1}}^{S_n} M_{n-1} \m M_n.  \]  Vanishing of $\rH_0^{\FI}(M)$ measures the generation degree of $M$ and vanishing of both $\rH_0^{\FI}(M)$ and its first derived functor $\rH_1^{\FI}(M)$ measure the presentation degree of $M$. These derived functors extend to complexes of $\FI$-modules $M_{\bullet}$ in the standard way; explicitly, the $\FI$-homology $\HH_k(M_\bullet)$ is computed by replacing $M_\bullet$ by a quasi-isomorphic complex of projective $\FI$-modules, applying $\rH_0^{\FI}$, then taking homology of the resulting complex. This should not be confused with $\rH_k(M_{\bullet})$ which denotes the usual homology of the complex $M_\bullet$.

\begin{atheorem} 
\label{athm:growth-bounds}
Let $M_{\bullet}$ be a complex of $\FI$-modules. Suppose $\HH_k(M_{\bullet})_n$ vanishes for $n$ larger than a linear function of $k$. Then we have:
 \begin{enumerate}
\item The stable degree of $\rH_k(M_{\bullet})$ grows at most linearly in $k$.
\item The local degree, generation degree, and presentation degree of $\rH_k(M_{\bullet})$ grow at most quadratically in $k$. 
\end{enumerate}
\end{atheorem}

See Theorem~\ref{thm:growth-bounds} for a quantitative version of this theorem. The key example we apply these results to is when $M_\bullet=C_\bullet \Gamma$ is the chains on an $\FI$-group $\Gamma$, so that $\rH_k(M_{\bullet})$ is the group homology $\rH_k(M_{\bullet})=\rH_k(\Gamma)$.
In particular, we will apply the quantitative version of Theorem \ref{athm:growth-bounds} (Theorem \ref{thm:growth-bounds}) to congruence subgroups of general linear groups. Given an ideal $I$ in a ring $R$, let $\GL_n(R,I)$ denote the kernel of $\GL_n(R) \m \GL_n(R/I)$. This is called the \emph{level-$I$ congruence subgroup} of $\GL_n(R)$. The groups $\{\GL_n(R,I)\}$ assemble to form an $\FI$-group $\GL(R,I)$ whose homology groups form $\FI$-modules. Theorem \ref{athm:growth-bounds} gives the following.


\begin{aapplication}
\label{athm:congruence-subgroup}
Let $I$ be an  ideal in a ring $R$ satisfying Bass's stable range condition $\SR_{d+2}$. Then we have: \begin{enumerate}
\item\label{b1} The stable degree of $\rH_k(\GL(R,I))$ is  $ \le 2k+d$.
\item\label{b2} The local degree of $\rH_k(\GL(R,I))$ is $ \le 2 k^2 +2(d+2)k + 2(d+1)$.
\item\label{b3} The generation degree of $\rH_k(\GL(R,I))$ is  $\le 2 k^2 + (2d + 6)k + 3 (d+1)$.
\item\label{b4} The presentation degree of $\rH_k(\GL(R,I))$ is $ \le 4 k^2 + (4d + 10) k + 5d +6$. 
\end{enumerate} The same bounds hold for $\rH_k(\GL(R,I);\bk)$ with any coefficients. In particular, we have: \begin{enumerate}[label=(\alph*)]
\item\label{polycong} Let $\bk$ be a field and assume that $\dim_{\bk} \rH_k(\GL_n(R,I); \bk)$ is finite. Then there are polynomials $p^{R,I}_{k,\bk}$ of degree at most $2k+d$  such that $$\dim_{\bk} \rH_k(\GL_n(R,I); \bk) = p^{R,I}_{k,\bk}(n)$$ if $n >  2 k^2 +2(d+2)k + 2(d+1)$.
\item\label{polycoeq} The natural map $$\Ind_{S_{n-1}}^{S_n} \rH_k(\GL_{n-1}(R,I)) \to \rH_k(\GL_n(R,I))$$ is surjective for $n > 2 k^2 + (2d + 6)k + 3 (d+1)$, and the kernel of this map is the image of the difference of the two natural maps $$\Ind_{S_{n-2}}^{S_n} \rH_k(\GL_{n-2}(R,I)) \rightrightarrows \Ind_{S_{n-1}}^{S_n} \rH_k(\GL_{n-1}(R,I))$$ for $n >  4 k^2 + (4d + 10) k + 5d +6$.
\end{enumerate} 

\end{aapplication}

See \cite[Definition 2.19]{Ba} for a definition of Bass' stable range condition. Recall  that any $d$-dimensional Noetherian ring satisfies Bass's stable range condition $\SR_{d+2}$. In particular, fields satisfy $\SR_{2}$ (with $d=0$) and Dedekind domains satisfy $\SR_{3}$ (with $d=1$).  We note that the finiteness condition in Part~(a) of Application~\ref{athm:congruence-subgroup} is satisfied for many classes of ideals, including all ideals in rings of integers in number fields. 

The previously best known stable range for congruence subgroups is due to Church--Ellenberg \cite[Theorem C$'$]{castelnuovo-regularity} who established an exponential stable range. Their result improved upon the work of Putman \cite{Pu} who also established an exponential range with coefficients in a field with characteristic large compared to the homological degree, and the work of Church--Ellenberg--Farb--Nagpal~\cite{fi-noeth} who established an integral result but with no explicit stable range at all. We note that after the release of an initial draft of this paper, Gan and Li established a linear stable range for congruence subgroups. \cite[Theorem~5]{gan-li}. 

Djament conjectured that the stable degree of $\rH_k(\GL(R,I); \bk)$ is $ \le 2k$ in \cite[Conjecture~1]{djament2} (also see  \cite[\S 5.2]{djament-vespa} for further discussion).  Part~(1) of Application~\ref{athm:congruence-subgroup} proves that this stable degree is $\le 2k+d$. Thus, up to an additive constant, Application~\ref{athm:congruence-subgroup} establishes this conjecture. The conjecture was subsequently established by Djament \cite[Theorem 2]{djament3}.

Theorem \ref{athm:growth-bounds} can also be used in other contexts. For example, it can be used to improve the ranges in Patzt and Wu's theorem on the homology of Houghton groups \cite[Theorem B]{PW}.

\subsection*{The complex of mod-$I$ split partial bases}
We prove Application~\ref{athm:congruence-subgroup} by connecting the $\FI$-homology of $\GL(R,I)$ with the \emph{complex of mod-$I$ split partial bases $\SPB_n(R,I)$}. This is a simplicial complex whose  maximal simplices correspond to bases for $R^n$ that are congruent mod $I$ to the standard basis, with the lower-dimensional simplices encoding bases for summands of $R^n$ together with a complement; see Definition~\ref{def:SPB-modI} for a precise definition. (Experts should note that $\SPB_n(R,I)$ may be slightly different than the complex one has in mind; the key feature that distinguishes our definition is that every simplex of $\SPB_n(R,I)$ belongs to an $(n-1)$-simplex.) Together these form an $\FI$-simplicial complex $\SPB(R,I)$ with an action of $\GL(R,I)$, so its $\GL(R,I)$-equivariant homology forms an $\FI$-module.

\begin{atheorem}
\label{athm:CGammaSPB}
For any ring $R$ and any proper ideal $I\subset R$, the $\FI$-homology of the chains on the congruence $\FI$-group $\GL(R,I)$ is computed by the $\GL(R,I)$-equivariant homology of the complex of mod-$I$ split partial bases $\SPB(R,I)$: \[\HH_k(C_\bullet\GL(R,I))\iso \widetilde{\rH}_{k-1}^{\GL(R,I)}(\SPB(R,I))\] for all $k\geq 0$, and similarly for any coefficient group $\bk$.
\end{atheorem}

Theorem~\ref{athm:CGammaSPB} tells us that to apply Theorem~\ref{athm:growth-bounds} to $\GL(R,I)$, we must bound the homology of $\SPB(R,I)$. Fortunately, Charney studied closely related complexes in \cite{Charney}; her results imply that  that these complexes $\SPB_n(R,I)$ are acyclic in dimensions up to $\frac{n-d-3}{2}$. A surprising consequence of Application~\ref{athm:congruence-subgroup} and Theorem~\ref{athm:CGammaSPB} is that we can prove that Charney's result is very close to \emph{sharp}, at least in certain cases (and probably in many more).
\begin{atheorem} \label{athm:connectivity}
Given any $\ell>0$, for each $k>0$ we have \[\widetilde{\rH}_{k-1}(\SPB_{2k}(\Z/p^\ell,p);\F_p)\neq 0.\]
Given any number ring $\cO$ and any prime power $p^a>2$, for each $k>0$ we have
\[\text{\textrm{either\ \  }} \widetilde{\rH}_{k-1}(\SPB_{2k}(\cO,p^a);\F_p)\neq 0\qquad\text{or}\qquad\widetilde{\rH}_{k-1}(\SPB_{2k+1}(\cO,p^a);\F_p)\neq 0.\]
\end{atheorem}
Note that Charney's bound implies all these complexes are $(k-2)$-acyclic, since these rings satisfy $\SR_{3}$, so these are the first nonzero homology groups. 
We prove Theorem~\ref{athm:connectivity} by using known results of Browder--Pakianathan, Lazard, and Calegari to prove the bounds in Application~\ref{athm:congruence-subgroup} are sharp. We then argue that if these complexes were \emph{more} acyclic, we could obtain even \emph{stronger} bounds in Application~\ref{athm:congruence-subgroup}, contradicting these known results. As the proof of Theorem~\ref{athm:connectivity} shows, all that is necessary for a theorem like this is a lower bound $\dim \rH_k(\GL_n(R,I);\F)\ \gg_n\  n^{2k-1}$ for a given $k$ and field $\F$. The restrictions on the rings and ideals here are not essential to the argument; we use them only to deduce such a lower bound from the literature. It is therefore likely that Theorem~\ref{athm:connectivity} holds in  greater generality. 


\subsection*{Outline of the paper} 


In \S \ref{sec:preliminaries}, we recall some basic facts about $\FI$-modules. In \S \ref{sec:local-and-stable-degree}, we prove several properties of stable degree and local degree including that they can be used to bound presentation degree. In \S \ref{secA}, we study \TypeA{} spectral sequence arguments and apply our results to configuration spaces. In \S \ref{secB}, we study \TypeB{} spectral sequence arguments and apply our results to congruence subgroups. In \S \ref{secFIgroup}, we show that the $\FI$-homology of the chains on an $\FI$-group is given by the equivariant homology of a natural $\FI$-simplicial complex. In \S \ref{secSPB}, we identify this simplicial complex in the case of congruence subgroups with the complex of mod-$I$ split partial bases, and prove Theorems~\ref{athm:CGammaSPB} and \ref{athm:connectivity}.





\subsection*{Acknowledgements} The first author is grateful to Mladen Bestvina and Andrew Putman for conversations regarding the homology of complexes of split partial bases. The third author would like to thank Andrew Snowden and Steven Sam for several useful conversations on local cohomology of $\FI$-modules. We are grateful to them for allowing us to include Proposition~\ref{prop:delta}, which originated in joint work of Nagpal, Snowden, and Sam. We thank the anonymous referee for a careful reading of the paper.

\section{Preliminaries on  $\FI$-modules}
\label{sec:preliminaries}
In this section, we review some basic definitions, constructions, and results concerning $\FI$-modules. For more background, see \cite{fimodules}. Also see \cite{symc1} for a discussion of $\FI$-modules from the perspective of twisted commutative algebras. The primary new results in this section are Theorem~\ref{thm:shift-of-homology} and Proposition~\ref{prop:delta}.

\subsection{Induced and semi-induced $\FI$-modules} 
Recall that $\FI$ denotes the category of finite sets and injections. Similarly let $\FB$ denote the category of finite sets and bijections. For any category $\mathcal C$, the term $\mathcal C$-module will mean a functor from $\mathcal C$ to the category of abelian groups and we denote the category of $\mathcal C$-modules by $\Mod_\mathcal C$. Similarly, the term $\mathcal C$-group will mean a functor from $\mathcal C$ to the category of groups.

There is a forgetful functor $\Mod_{\FI} \m \Mod_{\FB}$ and we denote its left adjoint by $\cI:\Mod_{\FB} \m \Mod_{\FI}$. This can be described concretely as follows. Given an $\FI$-module or $\FB$-module $M$, let $M_S$ denote its value on a set $S$ and let $M_n$ denote its value on the standard set of size $n$, $[n] \coloneq \{1, 2, \ldots, n\}$. For an $\FB$-module $V$, we have that:  \[ \cI(V)_S \cong \bigoplus_{n \ge 0} \Z[\Hom_{\FI}([n], S)]\otimes_{\Z[S_n]} V_n.   \] We call $\FI$-modules of the form $\cI(V)$ \emph{induced}; see \cite[Definition~2.2.2]{fimodules} for more details (note that there the notation $M(V)$ is used in place of $\cI(V)$ and they call these modules $\FI \sharp$ instead of induced). The category of $\FI$-modules has enough projectives. The projective $\FI$-modules are exactly the modules of the form $\cI(V)$ with each $V_n$ projective as a $\mathbb Z[S_n]$-module; see \cite[Proposition~2.3.10]{weibel} and \cite[Corollary~9.40]{transformation}.

We say that an $\FI$-module is \emph{semi-induced}\footnote{These were previously been called $\sharp$-filtered $\FI$-modules by Nagpal~\cite{nagpal}. A very similar construction under the name $J$-good functors appeared in Powell~\cite{powell}.} if it admits a finite length filtration where the quotients are induced modules. 
The following is a useful property of semi-induced modules which holds not only for $\FI$-modules but also many other similar functor categories; see for example \cite[Remark~2.33]{ramos} and \cite[Corollary~4.23]{vimod}.

\begin{proposition}
\label{prop:semi-induced-ses} In a short exact sequence of $\FI$-modules, if two of the terms are semi-induced, then so is the third.
\end{proposition}
\begin{proof} Two of the three cases are proven in \cite[Proposition~A.8, Theorem~A.9]{djament}. The remaining case is an  immediate corollary of \cite[Theorem~B]{ramos}. 
\end{proof}

\subsection{$\FI$-homology}
\label{subsec:koszul-intro}

Any $\FB$-module can be upgraded to an $\FI$-module by declaring that all injections $S \to T$ which are not bijections act as the zero map. This assignment gives a functor $\FBMod \to \FIMod$ which admits a left adjoint that we will denote by $\rH^{\FI}_0$ and call \emph{$\FI$-homology}. Since $\rH^{\FI}_0$ admits a right adjoint, it is right-exact. We denote the total left-derived functor of  $\rH^{\FI}_0$ by $\rL \rH^{\FI}$ and denote the $i$th left-derived functor by $\rH_i^{\FI}(M)$. Often, we will consider $\rH_i^{\FI}(M)$ as an $\FI$-module by post-composing with the functor $\FBMod \to \FIMod$ described above.


\begin{definition}
The \emph{degree} of a non-negatively graded abelian group $M$ is the smallest integer $d \ge -1$, denoted $\deg M$,  such that  $M_n = 0$ for $n>d$. Evaluating an $\FI$-module or $\FB$-module $M$ on the standard sets $[n]$ gives a non-negatively graded abelian group, and so we can make sense of $\deg M$. Let $t_i(M) \coloneq \deg \rH_i^{\FI}(M)$. We call $t_0(M)$ the \emph{generation degree} of $M$ and call $\max(t_0(M),t_1(M))$ the \emph{presentation degree} of $M$. We say that an $\FI$-module $M$ is \emph{presented in finite degrees} if $t_0(M)<\infty$ and $t_1(M)<\infty$.
\end{definition}
The generation degree and presentation degree as in our definition above bounds the smallest possible degrees of generators and relations in any presentation for $M$ \cite[Proposition~4.2]{castelnuovo-regularity}. 

\begin{theorem}
\label{thm:homology-acyclics}
We have the following: \begin{enumerate}[label={(\arabic*)}]
\item\label{part:abeliancategory} The category of $\FI$-modules presented in finite degrees is abelian; in other words, for any map between $\FI$-modules presented in finite degrees, the kernel and cokernel are also presented in finite degrees.
\item An $\FI$-module presented in finite degrees is $\FI$-homology acyclic if and only if it is semi-induced.
\end{enumerate}
\end{theorem}


\begin{proof} The first statement is an immediate corollary of \cite[Theorem~A]{castelnuovo-regularity} (or see  \cite[Theorem~B]{ramos-coh} for more details) and the second statement is  \cite[Theorem~B]{ramos}.
\end{proof}

\begin{proposition}
\label{prop:equalizer}
Let $M$ be an $\FI$-module. \begin{enumerate}
\item Then $t_0(M) \leq d$ if and only if $\Ind_{S_{n-1}}^{S_n} M_{n-1} \m M_n$ is surjective for $n>d$. 	
\item Then $t_1(M) \leq r$ if and only if the kernel of $\Ind_{S_{n-1}}^{S_n} M_{n-1} \m M_n$ is the image of the difference of the two natural maps $\Ind_{S_{n-2}}^{S_n} M_{n-2} \rightrightarrows \Ind_{S_{n-1}}^{S_n} M_{n-1}$.		
\end{enumerate}	
\end{proposition}
\begin{proof} The first statement follows from the definition of $\rH_0^{\FI}$ given in the introduction. 		

Let $\epsilon$ denote the sign representation of $S_2$. It follows from \cite[Proposition 5.10]{castelnuovo-regularity} that $\rH_1^{\FI}(M)_n$ is equal to the homology of the chain complex:  \[\Ind_{S_{n-2} \times S_2 }^{S_n} M_{n-2} \boxtimes \epsilon \m \Ind_{S_{n-1}}^{S_n} M_{n-1} \m M_n \] Since the difference of the two natural maps \[\Ind_{S_{n-2}}^{S_n} M_{n-2} \rightrightarrows \Ind_{S_{n-1}}^{S_n} M_{n-1}\] factors through the map \[ \Ind_{S_{n-2} \times S_2}^{S_n} M_{n-2} \boxtimes \epsilon \m \Ind_{S_{n-1}}^{S_n} M_{n-1}\] and \[ \Ind_{S_{n-2}}^{S_n} M_{n-2} \m \Ind_{S_{n-2} \times S_2}^{S_n} M_{n-2} \boxtimes \epsilon\] is surjective, $\rH_1^{\FI}(M)_n$ is also isomorphic to the homology of the chain complex: \[\Ind_{S_{n-2}}^{S_n} M_{n-2} \rightrightarrows \Ind_{S_{n-1}}^{S_n} M_{n-1} \m M_n \] The claim now follows. 
\end{proof}



As with any derived functor, $\FI$-homology extends to any bounded-below complex of $\FI$-modules.
If $M_{\bullet}$ is a bounded-below complex of $\FI$-modules, we write $\HH_i(M_{\bullet})$ for the ``$\FI$-hyper-homology'' computed by replacing $M_{\bullet}$ by a quasi-isomorphic complex of projective (or just semi-induced; Theorem~\ref{thm:homology-acyclics}~(2)) $\FI$-modules, applying $\rH^{\FI}$ term-wise, then taking $\ker$/$\im$ in homological degree $i$. Similarly, we write $\bt_i(M_{\bullet})$ for $\deg \HH_i(M_{\bullet})$. 

A reason we write $\HH_i(M_{\bullet})$ and $\bt_i(M_{\bullet})$, rather than just $\rH_i^{\FI}(M_{\bullet})$ and $t_i(M_{\bullet})$, is to emphasize that $M_{\bullet}$ is a complex of $\FI$-modules. Another reason for using this notation is that $\rH_i^{\FI}(M_{\bullet})$ could plausibly mean the complex of $\FI$-modules obtained by applying the functor $\rH_i^{\FI}$ to each $\FI$-module $M_k$ individually (ignoring the differential on $M_\bullet$) and then reassembling these groups back into a chain complex; we will never use this notation or this notion.



For any complex of $\FI$-modules $M_{\bullet}$, we will denote the homology of the complex by $\rH_i(M_{\bullet})\coloneq \ker(M_i\to M_{i-1})/\im(M_{i+1}\to M_i)$. To avoid confusion with $\FI$-homology, throughout this paper we observe the convention that the notation $\rH_i(M_{\bullet})$ \emph{without} superscript $\FI$ always refers to $\ker$/$\im$; any functor obtained from $\rH^{\FI}_0$ will always have the superscript $\FI$.

	

\subsection{Shifts, derivatives, and FI-homology}
\label{sec:shift-homology}
For a finite set $S$, we have a natural transformation $\tau_{S}$ given by the composite \[ \FI \to \FI \times \FI \to \FI\] where the first transformation takes $T \mapsto (S,T)$ and the second transformation is given by the disjoint union. Pulling back along $\tau_{S}$   defines a natural transformation on $\Mod_{\FI}$ which depends (up to an isomorphism) only on the size of $S$. We define the shift functor $\Sigma \colon \Mod_{\FI} \to \Mod_{\FI}$ to be this functor when $S=\{\star\}$, and we define $\Sigma^n $ to be the $n$-fold iterate of $\Sigma$.  For any $\FI$-module $M$, we have $\deg (\Sigma M) =\deg M - 1$ (unless $M = 0$, in which case $\deg (\Sigma M) = \deg  M = -1$).

The natural transformation $\id \to \tau_S$ induces a natural transformation $\id \to \Sigma$ whose cokernel will be denoted by $\Delta$ (note that $\Sigma$ and $\Delta$ are denoted by $S$ and $D$ respectively in \cite{fimodules,fi-noeth,castelnuovo-regularity}). 
Induced modules (and hence semi-induced modules) are acyclic with respect to $\Delta$ \cite[Corollary~4.5]{castelnuovo-regularity}. Moreover, if $V$ is an $\FB$-module then the short exact sequence \[ 0 \to \cI(V)  \to \Sigma \cI(V) \to \Delta \cI(V) \to 0 \] splits, and we have $\Delta\cI(V)=\cI(\Sigma V)$ \cite[Lemma~4.4]{castelnuovo-regularity}. It follows that $\Delta$ takes semi-induced modules to semi-induced modules.

It is well known that $t_0(\Sigma M)\leq t_0(M)$ (see \cite[Corollary~2.13]{fi-noeth}). A key ingredient in the proof of Theorem \ref{athm:growth-bounds} is the following derived version of this statement. 


\begin{theorem}
	\label{thm:shift-of-homology}
	Let $M_{\bullet}$ be a bounded-below graded complex of $\FI$-modules. Then we have $\bt_i(\Sigma M_{\bullet})\leq \bt_i(M_{\bullet})$ for all $i$.
\end{theorem} 
\begin{proof}
	
	The key to this theorem is a lemma, due to Church~\cite{church-bounding}, which leads to a natural long exact sequence (this is a derived version of the long exact sequence in \cite[Theorem~1]{gan-sequence})
	\begin{equation}
		\label{eq:FIhomologyLES}
		\ldots \to  \HH_i(M_{\bullet}) \to \HH_i(\Sigma M_{\bullet}) \to \Sigma \HH_i(M_{\bullet}) \to \ldots.
	\end{equation} We explain the construction of this long exact sequence below, but first we note that the assertion of the theorem follows immediately from it: \[ \bt_i(\Sigma M_{\bullet}) = \deg \HH_i(\Sigma M_{\bullet})  \leq \max(\deg \HH_i(M_{\bullet}), \deg \Sigma \HH_i(M_{\bullet}) ) =  \bt_i(M_{\bullet}) . \]
	
	
	To construct the long exact sequence above, start by replacing $M_{\bullet}$ with a quasi-isomorphic complex $P_{\bullet}$ of projective $\FI$-modules. We then get a split short exact sequence
	\[0 \to P_{\bullet} \to \Sigma P_{\bullet} \to \Delta P_{\bullet} \to 0\]
	which induces a long exact sequence in $\FI$-homology
	\[ \ldots \to  \HH_i(P_{\bullet}) \to \HH_i(\Sigma P_{\bullet}) \to \Sigma \HH_i(P_{\bullet}) \to \ldots. \]
	By definition, the first term is $\HH_i(P_\bullet)=\HH_i(M_{\bullet})$, and since $\Sigma$ is exact, the second term is $\HH_i(\Sigma P_{\bullet})=\HH_i(\Sigma M_{\bullet})$.
	Since $\Delta$ takes projectives to projectives \cite[Lemma~4.7(iv)]{castelnuovo-regularity}, the third term is \[\HH_i(\Delta P_\bullet)=\rL_i(\rH^{\FI}\Delta)(M_{\bullet})\iso\rL_i(\Sigma\rH^{\FI})(M_{\bullet})=\Sigma\rL_i(\rH^{\FI})(M_{\bullet})=\Sigma\HH_i(M_{\bullet}).\] In the second equality, we used the isomorphism $\Sigma \rH^{\FI}\iso \rH^{\FI}\Delta$ of \cite{church-bounding}, and in the third we used that $\Sigma$ is exact. Therefore this is the desired long exact sequence.
\end{proof}

\subsection{The shift theorem and stable degree}
The following  theorem is a slight generalization of Nagpal's structure theorem for finitely generated $\FI$-modules to the case of $\FI$-modules with finite presentation degree.

\begin{theorem}[{\cite[Proposition~6.4 and Theorem~A.9]{djament} or \cite[Theorem~C]{ramos}}]
\label{thm:shift-theorem}
Let $M$ be an $\FI$-module presented in finite degree. Then for large enough $n$, $\Sigma^n M$ is semi-induced.
\end{theorem}


\begin{definition} Let $M$ be an $\FI$-module. We say that an element $x \in M(S)$ is \emph{torsion} if there exists an injection $f \colon S \to T$ such that $f_{\ast}(x) =0$. An FI-module is \emph{torsion} if it consists entirely of torsion elements.
\end{definition} 

\begin{definition}
We define the \emph{stable degree} of an $\FI$-module $M$, denoted $\delta(M)$, to be the least number $n \ge -1$ such that $\Delta^{n+1} M$ is torsion.
\end{definition} The notion of  stable degree was introduced in \cite{djament-vespa} where it was called weak degree.  We summarize below some properties of stable degree for $\FI$-modules presented in finite degrees. Before this project started, Steven Sam, Andrew Snowden and the third author 
worked out a proof of these properties in a private communication. We are grateful to Steven Sam and Andrew Snowden for allowing us to include these here. 
\begin{proposition}
\label{prop:delta} 
Let $K, L, M$ and $N$ be $\FI$-modules presented in finite degrees.
\begin{enumerate}[label={(\arabic*)}]
\item \label{part:deltat0semi} If $M$ is semi-induced, then $\delta(M) = t_0(M)$.
\item \label{part:deltaSigma} $\delta(M) = \delta(\Sigma^n M)$ for any $n \ge 0$.
\item \label{part:deltadegSigma} $\delta(M)$ is the common value of $t_0(\Sigma^n M)$ for $n \gg 0 $. 
\item \label{part:deltat0gen} $\delta(M) \le t_0(M)<\infty$.
\item \label{part:LMN} If $ 0 \to L \to M \to N \to 0$ is a short exact sequence, $\delta(M) = \max (\delta(L), \delta(N)).$
\item \label{part:subquotient} If $K$  is a subquotient of $M$,  $\delta(K) \le \delta(M).$
\item \label{part:iterDelta} The cokernel $\iterDelta_a M$ of the natural map $M \to \Sigma^a M$  for $a>0$ satisfies $\delta(\iterDelta_a(M)) = \max({\delta(M) - 1}, -1)$.
\end{enumerate}
\end{proposition}
\begin{proof} 
 Part~\ref{part:deltat0semi}: First suppose $M=\cI(V)$ is induced. From the equality $\Delta\cI(V)=\cI(\Sigma V)$, we have $\Delta^k\cI(V)=\cI(\Sigma^k V)$, and the smallest $n$ such that $\Sigma^{n+1}V=0$ is $\deg V=t_0(M)$. This shows that $\delta(\cI(V)) = t_0(\cI(V))$. Since induced modules are acyclic with respect to both $\rH^{\FI}_0$ (Theorem~\ref{thm:homology-acyclics}) and $\Delta$ (\cite[Corollary~4.5]{castelnuovo-regularity}), we conclude that the result holds for semi-induced modules as well.  Part~\ref{part:deltaSigma} follows from the fact that $\Delta$ commutes with $\Sigma^n$ \cite[Proposition~1.4]{djament-vespa}, and the fact that $T$ is torsion if and only if $\Sigma^n T$ is torsion. Part~\ref{part:deltadegSigma}: Once $n$ is large enough that $\Sigma^n M$ is semi-induced (Theorem~\ref{thm:shift-theorem}), this follows immediately from Part~\ref{part:deltat0semi} and Part~\ref{part:deltaSigma}.
Part~\ref{part:deltat0gen} follows from Part~\ref{part:deltadegSigma} in light of the fact that $t_0(\Sigma^n M) \le t_0(M)$ (e.g.\ \cite[Corollary~2.1]{fi-noeth}). 

Part~\ref{part:LMN}: Choose $n$ large enough that  $\Sigma^n L$, $\Sigma^n M$, and $\Sigma^n N$ are semi-induced. Since semi-induced modules are homology-acyclic, we have a short exact sequence \[0\to \rH^{\FI}_0(\Sigma^n L)\to \rH^{\FI}_0(\Sigma^n M)\to \rH^{\FI}_0(\Sigma^n N)\to 0.\] Thus, $t_0(\Sigma^n M)=\max(t_0(\Sigma^n L),t_0(\Sigma^n L))$, which implies the claim in light of Part~\ref{part:deltadegSigma}.  Part~\ref{part:subquotient} is a consequence of Part~\ref{part:LMN}.  Part~\ref{part:iterDelta}: By \cite[Proposition~1.4 (7)]{djament-vespa}, it suffices to prove the result when $a = 1$. In this case $Q_a(M) $ is just $\Delta M$, and hence by definition of $\delta$ we have $\delta(\iterDelta_a(M)) = \max({\delta(M) - 1}, -1)$. This completes the proof.
\end{proof}

\subsection{Local cohomology and local degree}
Let $\Gamma_{\fm}(M)$ denote the maximum torsion submodule contained in $M$.  The functor $\Gamma_{\fm}$ is left-exact, and so we can consider its right-derived functor $\rR \Gamma_{\fm}$. We also write $\rH^i_{\fm}$ for $\rR^i \Gamma_{\fm}$, and call these functors \emph{local cohomology} (this terminology is chosen because of its similarity to the classical notion of local cohomology from commutative algebra). We write $h^i(M)$ for $\deg \rH^i_{\fm}(M)$.  


The following result is a strengthening of \cite[Theorem~E]{li-ramos}.



\begin{theorem} \label{thm:local-cohomology}
Let $M$ be an $\FI$-module presented in finite degrees. Then there exists a complex \[ 0 \to I^0 \to I^1 \to \ldots \to I^N \to 0\] exact in all high enough degrees such that $I^0 = M$ and $I^i$ is semi-induced for $i > 0$.
Moreover, for any such complex the following holds.
\begin{enumerate}
\item $\rH^i(I^{\bullet}) = \rH^i_{\fm}(M)$.
\item $h^0(M) \le \min(t_0(M), t_1(M))+t_1(M)-1$.
\item $h^1(M) \le \delta(M) + t_0(M)-1$.
\item $h^i(M) \le 2 \delta(M)-2(i-1)$ for $i \ge 2$. In particular, $\rH^i_{\fm}(M) = 0$ if $i > \delta(M) + 1$.
\end{enumerate}
\end{theorem}
\begin{proof}
Note that the assumption that the complex $I^\bullet$ is exact in high enough degree is equivalent to saying that $\rH^i(I^\bullet)$ is torsion for all $i$.
The existence of such a complex is proven in \cite[Theorem~A]{nagpal}. Part~(1) is proven in \cite[Theorem~E]{li-ramos} (also see \cite[Theorem~4.10]{ramos-coh}). Part~(2) is \cite[Corollary~F]{castelnuovo-regularity}.\footnote{The proof of \cite[Corollary~F]{castelnuovo-regularity} has been greatly simplified by Church~\cite{church-bounding}, based on a inductive argument given by Li~\cite[Theorem~1.3]{li}.} For the remaining parts, note that for all $i>0$, we have \[\rH^i_{\fm}(M)= \ker(\coker(I^{i-1}\to I^i) \to I^{i+1}) = \Gamma_{\fm}(\coker(I^{i-1}\to I^i))  \] where the last equality follows from the fact that $I^{i+1}$ is torsion-free and that $\rH^i_{\fm}(M)$ is torsion. The statement of \cite[Corollary~F]{castelnuovo-regularity} is that the degree of torsion in such a cokernel can be bounded:
\[h^i(M)=\deg \Gamma_{\fm}(\coker(I^{i-1}\to I^i))\leq t_0(I^{i-1})+t_0(I^i)-1.\] Note that for $i>0$, we have $t_0(I^i)=\delta(I^i)$. Therefore to obtain the claimed bounds on $h^i(M)$, it suffices to show that the complex can be chosen so that $\delta(I^i) \le \delta(M) -i + 1$ for $i > 0$. We do this by induction of $\delta(M)$ as follows.

In the base case $\delta(M) = -1$, we can choose $I^i = 0$ for $i>0$. Now assume $\delta(M) > 0$, and choose an $n$ large enough that $\Sigma^n M$ is semi-induced (Theorem~\ref{thm:shift-theorem}). Let $M'$ be the cokernel of $M \to \Sigma^n M$. By Proposition~\ref{prop:delta}\ref{part:iterDelta}, we see that $\delta(M') < \delta(M) $. By induction, there is a complex \[ 0 \to J^0 \to J^1 \to \ldots \to J^{N'} \to 0\] such that $J^0 = M'$ and $\delta(J^i) \le \delta(M') - i + 1$ for $i>0$. Now set $I^0 = M$,  $I^1 = \Sigma^n M$ and $I^{i+1} = J^i$ for $i>1$, and observe that $I^{\bullet}$ has the desired property.
\end{proof}

\begin{remark}
	Sometimes it is possible to improve the bounds on $h^i(M)$ is the above proposition. By combining \cite[Theorem~1.1]{regthm}  and \cite[Theorem~A]{castelnuovo-regularity}, we see that \[ h^i(M) \le t_n(M) -n -i \le t_0 + t_1 -1 - i.\] Thus, if $\delta(M)$ is large compared to $t_1(M)$, these bounds are better than the ones in the theorem above.
\end{remark}

%
%

\begin{definition}
Let $M$ be an $\FI$-module. We define the \emph{local degree} of $M$ to be the quantity $h^{\max}(M) \coloneq \max_{i\geq 0} h^i(M)$.
\end{definition}

Recall that Theorem~\ref{thm:shift-theorem} tells us that a sufficiently high shift $\Sigma^nM$ of any $\FI$-module presented in finite degree will be semi-induced. The following corollary tells us that the local degree quantifies precisely how much we need to shift $M$ for this to happen.

\begin{corollary} [{Li--Ramos \cite[Theorem F, Part~(2)]{li-ramos}}] 
\label{cor:hmax}
Let $M$ be an $\FI$-module presented in finite degrees. Then $\Sigma^n M$ is semi-induced if and only if $n > h^{\max}(M)$. In particular, $\rR \Gamma_{\fm}(M) = 0$ if and only if $M$ is semi-induced.
\end{corollary}
\begin{proof} Let $I^{\bullet}$ be the complex constructed in the previous theorem. If $n > h^{\max}(M)$, then $\Sigma^n I^{\bullet}$ is exact (shifts commute with local cohomology). Since a shift of a semi-induced module is semi-induced, we see that $\Sigma^n I^i$ is semi-induced for $i > 0$. By Proposition~\ref{prop:semi-induced-ses}, we conclude that $\Sigma^n I^0 = \Sigma^n M$ is semi-induced. Conversely, if $n \le h^{\max}(M)$, then $\rH^i_{\fm}(\Sigma^n M)$ is nonzero for some $i$ (shifts commute with local cohomology). However, $\Sigma^n M$ cannot be semi-induced by the theorem above.
\end{proof}

For any ring $\bk$, we say that an $\FI$-module $M$ is an $\FI$-module over $\bk$ if the functor $M\colon \FI \to \Mod_{\Z}$ factors through $\Mod_{\bk}\to \Mod_{\Z}$.
\begin{proposition}
\label{prop:over-fields}
Suppose $\bk$ is a field, and let $M$ be an $\FI$-module over $\bk$ which is presented in finite degrees and with $M_n$ finite dimensional for all $n$. Then there exists an integer-valued polynomial $p \in \mathbb{Q}[X]$ of degree $\delta(M)$ such that $\dim_{\bk} M_n = p(n)$ for $n > h^{\max}(M)$.
\end{proposition}
\begin{proof} First assume that $M = \cI(V)$ is an induced module. By the previous corollary, $h^{\max}(M) = -1$. And we have $\deg V = \delta(M)$. In this case the result follows from the following identity that holds for $n \ge 0$: \[\dim_{\bk} M_n = \sum_{j = 0}^{\deg V} \binom{n}{j} \dim_{\bk} V_j.\]
	
In general, let $N = h^{\max}(M) + 1$ and set $M' = \Sigma^N M$. It is enough to show that $\dim_{\bk} M'_n$ agrees with a polynomial of degree $\delta(M)$ for all $n \ge 0$. By Proposition~\ref{prop:delta}, we have $\delta(M') =  \delta(M)$, and by the previous corollary $M'$ is semi-induced. Then $M'$ admits a finite filtration such that the graded pieces are induced modules of stable degree at most $\delta(M)$ and at least one such graded piece is of stable degree exactly $\delta(M)$ (Proposition~\ref{prop:delta}\ref{part:LMN}). The result thus follows from the previous paragraph.
\end{proof}

\section{Properties of stable degree and local degree} 
\label{sec:local-and-stable-degree}
In this section, we show that the generation and presentation degrees of an $\FI$-module can be bounded linearly in terms of the stable and local degrees, and vice versa, and that \emph{together} they behave well under taking kernels and cokernels.

\begin{proposition}
\label{prop:tot-degree1}
Let $M$ be an $\FI$-module presented in finite degrees. Then we have the following: \begin{enumerate}[label={(\arabic*)}]
\item\label{bound:delta} $\delta(M)\leq t_0(M)$.
\item\label{bound:hmax} $h^{\max}(M)\leq t_0(M) + \max(t_0(M),t_1(M))-1$.
\item\label{bound:t0} $t_0(M) \le \delta(M) + h^{\max}(M) + 1$.
\item\label{bound:t1} $t_1(M) \le \delta(M) + 2 h^{\max}(M) + 2$.
\end{enumerate}
\end{proposition}
\begin{proof} Part~\ref{bound:delta} is Proposition~\ref{prop:delta}\ref{part:deltat0gen}. Part~\ref{bound:hmax} is obtained by combining the different cases of Theorem~\ref{thm:local-cohomology}, since $\delta(M)\leq t_0(M)$.

For the remaining parts, set $n = h^{\max}(M) + 1$.  Since $\Sigma^n M$ is semi-induced (Corollary~\ref{cor:hmax}), we have $t_0(\Sigma^n M) = \delta(\Sigma^n M) = \delta(M)$. This shows that (see \cite[Corollary~2.13]{fi-noeth}) \[t_0(M) \le t_0(\Sigma^n M) + n = \delta(M) + h^{\max}(M) + 1.\] This proves Part~\ref{bound:t0}. Part~\ref{bound:t1}: Consider a presentation \[ 0 \to K \to F \to M \to 0 \] with $F$ semi-induced and generated in  degrees $ \le t_0(M)$. By Proposition~\ref{prop:semi-induced-ses}, we see that $\Sigma^n K$ is semi-induced, since both $\Sigma^nF$ and $\Sigma^n M$ are. This shows that \[t_1(M) \le t_0(K) \le t_0(\Sigma^n K) + n = \delta(\Sigma^n K) + n.\]
By Proposition~\ref{prop:delta}\ref{part:subquotient}, we have $\delta(\Sigma^nK)\leq \delta(\Sigma^n F)$, and since $F$ is semi-induced we have $\delta(\Sigma^n F)=\delta(F) = t_0(F)\leq t_0(M)$. Combined with the bound  $t_0(M)\leq \delta(M)+n$ from Part~\ref{bound:t0}, we obtain
\[t_1(M)\leq \delta(\Sigma^n K)+n\leq t_0(M) + n \le \delta(M)+2n=\delta(M) + 2 h^{\max}(M) + 2.\qedhere\]
\end{proof}

\begin{proposition}
\label{prop:filtration}
Suppose the $\FI$-module $M$ has a finite filtration $M=F^0\supset \cdots\supset F^k=0$, and let $N^i=F^i/F^{i+1}$. Then $\delta(M)=\max_i \delta(N^i)$ and $h^{\max}(M)\leq \max_i h^{\max}(N^i)$.
\end{proposition}
\begin{proof}
The claim for $\delta$ follows from Proposition~\ref{prop:delta}\ref{part:LMN} by induction. Given $0\to L\to M\to N\to 0$, the long exact sequence $\cdots\to \rH^i_{\fm}(L)\to \rH^i_{\fm}(M)\to \rH^i_{\fm}(N)\to \cdots$ shows that $h^{\max}(M)$ is bounded by the maximum of $h^{\max}(L)$ and $h^{\max}(N)$. The proposition follows by induction.
\end{proof}

The following proposition is the key to analyzing \TypeA{} arguments. It gives us control over the stable and local degrees under taking kernels and cokernels.

\begin{proposition}
\label{prop:total-degree-ker-coker}
Let $f \colon A \to B$ be a map of $\FI$-modules presented in finite degrees. Then we have the following:
\begin{enumerate}[label={(\arabic*)}]
\item \label{part:deltaker} $\delta(\ker f) \le \delta(A)$.
\item \label{part:deltacoker}$\delta(\coker f) \le \delta(B)$.
\item \label{part:hmaxker}$h^{\max}(\ker f) \le \max(2 \delta(A) - 2, h^0(A), h^1(A), h^0(B)) \le \max(2 \delta(A) -2, h^{max}(A), h^{max}(B))$.
\item \label{part:hmaxcoker}$h^{\max}(\coker f) \le \max(2 \delta(A) - 2, h^{max}(A), h^{max}(B))$. 
\end{enumerate}
\end{proposition}

The proof of this proposition occupies the remainder of this section, but first we establish the following lemmas.

\begin{lemma}
Let $M^{\bullet}$ be a bounded complex of $\FI$-modules presented in finite degrees which is exact in high enough degree. 
Denote the cokernel of the map $\id \to \Sigma^a$ by $\iterDelta_a$. Then $\rH^i(\iterDelta_a(M^{\bullet}))$ vanishes in high enough degree.
\end{lemma}
\begin{proof} Choose $n$ large enough so that $\Sigma^n M^{\bullet}$ is an exact complex of semi-induced modules. Since semi-induced modules are acyclic with respect to $\iterDelta_a$ (\cite[Corollary~4.5]{castelnuovo-regularity}), we see that $\iterDelta_a \Sigma^n M^{\bullet}$ is exact. The result follows because  $\iterDelta_a$ commutes with $\Sigma^n$ (\cite[Proposition~1.4]{djament-vespa}).
\end{proof}


\begin{lemma}
Let $M^{\bullet}$ be a bounded complex of $\FI$-modules presented in finite degrees such that $\rH^i(M^{\bullet})$ is exact in high enough degree. Then there exists a double complex $I^{\bullet, \bullet}$  with the following properties: \begin{enumerate}[label=(\alph*)]
\item The first column $I^{0, \bullet}$ is  $M^{\bullet}$ and $I^{i,\bullet} = 0$ for $i<0$.
\item\label{part:rowsexact} The columns $I^{i, \bullet}$ are exact for $i > 0$.
\item $I^{i,j}$ is semi-induced if $i > 0$ and $\delta(I^{i,j}) \le \delta(M^j) - i +1$.
\item \label{part:colsalmostexact} Each row $I^{\bullet, j}$ is exact in high enough degree.
\end{enumerate}
Note that Part~\ref{part:colsalmostexact} implies $\rH^i(I^{\bullet, j})=\rH^i_{\fm}(M^j)$ by Theorem~\ref{thm:local-cohomology}.
\end{lemma}
\begin{proof} To build such a double complex, we proceed by induction on $d \coloneq \max_j \delta(M^j)$. If $d = -1$, then we can just take $I^{i, j} = 0$ for $i > 0$. Now suppose $d > -1$. Pick an $n$ large enough such that $\Sigma^n M^{\bullet}$ is an exact complex of semi-induced modules (see Theorem~\ref{thm:shift-theorem}). Set $I^{0,\bullet} =  M^{\bullet}$ and $I^{1,\bullet } = \Sigma^n M^{\bullet}$. The cokernel of the map $I^{0,\bullet} \to I^{1,\bullet }$ is $\iterDelta_n M^{\bullet}$. By the previous lemma, we see that $\rH^i(\iterDelta_n M^{\bullet})$ is torsion for each $i$. Thus, by induction on $d$, the theorem holds for the complex $\iterDelta_n M^{\bullet}$. Let $J^{i, j}$ be the corresponding double complex for $\iterDelta_n M^{\bullet}$. Set $I^{i+1, \bullet} = J^{i,\bullet}$ for $i >0$. It is now easy to check that $I^{\bullet, \bullet}$ has all the required properties. 
\end{proof}

\begin{proof}[Proof of Proposition~\ref{prop:total-degree-ker-coker}]Both Parts \ref{part:deltaker} and \ref{part:deltacoker} are special cases of Proposition~\ref{prop:delta}\ref{part:subquotient}. We now prove Parts \ref{part:hmaxker} and \ref{part:hmaxcoker}. For the ease of notation, denote the complex $0 \to \ker f \to A \to B \to \coker f \to 0$ by \[M^{\bullet} \colon 0 \to M^0 \to M^1 \to M^2 \to M^3 \to 0.\] Let $I^{\bullet, \bullet}$ be the complex as in the previous lemma. Since all the columns are exact, the spectral sequence corresponding to this double complex converges to 0. The first page of this spectral sequence is given by $E_1^{p,q} = \rH^p_{\fm}(M^q)$. Since this spectral sequence converges to zero, we see that the following must hold: \[h^i(M^0) \le \max_{j \le i } h^j(M^{i-j + 1}). \] In particular, $h^0(\ker f) \le h^0(A)$ and $h^1(\ker f) \le \max(h^1(A), h^0(B) )$. By Part~\ref{part:deltaker} and Theorem~\ref{thm:local-cohomology}, we see that $h^i(\ker f) \le 2 \delta(A) - 2(i-1)$ for $i \ge 2$. This proves Part~\ref{part:hmaxker}. Again, since $E_1^{p,q}$ converges to $0$, we observe that the following must hold:  \[h^i(M^3) \le \max_{j \ge i} h^j(M^{2 + i - j}).\] Thus for each $i$, we have \[h^i(\coker f) \le \max(h^i(B), h^{i+1}(A), h^{i+2}(\ker f)) \le \max(h^{\max}(B), h^{\max}(A), 2 \delta(A) - 2). \] This finishes the proof of Part~\ref{part:hmaxcoker}, and we are done.
\end{proof}

\section{\TypeA{} spectral sequence arguments and configuration spaces}
\label{secA}

In this section, we prove Theorem \ref{theorem:non-quillen} which establishes linear stable ranges in \TypeA{} spectral sequence arguments. We use this to prove our results on configuration spaces, Application \ref{ap:config}.

\subsection{The \TypeA{} setup} By a \TypeA{} setup, we mean a first quadrant spectral sequence $E_r^{p,q}$ of $\FI$-modules such that for some page $d$ we have bounds on $t_{0}(E_d^{p,q})$ and $t_{1}(E_d^{p,q})$ depending on $p$ and $q$. 


%

Theorem \ref{theorem:non-quillen} follows via Proposition~\ref{prop:tot-degree1} from the following proposition. Note that our spectral sequences are cohomologically indexed; Theorem~\ref{theorem:non-quillen} applies equally well to homologically indexed spectral sequences, but the precise bounds in Proposition~\ref{prop:non-quillen} would be slightly different.
\begin{proposition}
\label{prop:non-quillen}

Let $E_r^{p,q}$ be a cohomologically graded first quadrant spectral sequence of $\FI$-modules converging to $M^{p+q}$. Suppose that for some page $d$, the $\FI$-modules $E_d^{p,q}$ are presented in finite degrees, and set $D_{k}=\max_{p+q=k}\delta(E_d^{p,q})$ and $\eta_k=\max_{p+q=k}h^{\max}(E_d^{p,q})$.
Then we have the following:  \begin{enumerate}
\item $\delta(M^k) \le D_k$ 
\item $h^{\max}(M^k) \le \max\big(\,\max_{\ell \leq k+s-d} \eta_\ell,\ \ \max_{\ell \le 2k  -d +1} (2D_\ell-2)\,\big)$
\end{enumerate} where $s = \max(k+2, d)$.

\end{proposition}
\begin{proof} For all $r$, set $V_r^k=\bigoplus_{p+q=k} E_r^{p,q}$. Note that $D_k=\delta(V_d^k)$ and $\eta_k=h^{\max}(V_d^k)$. Since $M^k$ has a filtration whose associated graded is $V_\infty^k$, Proposition~\ref{prop:filtration} tells us that $\delta(M^k)=\delta(V_\infty^k)$ and $h^{\max}(M^k)\leq h^{\max}(V_\infty^k)$.

By definition, $V_{r+1}^k=\coker(V_r^{k-1}\to \ker(V_r^k\to V_r^{k+1}))$. Applying Proposition~\ref{prop:total-degree-ker-coker} shows $\delta(V_{r+1}^k)\leq \delta(V_r^k)$ and \[h^{\max}(V_{r+1}^k)\leq \max(2\delta(V_r^{k-1})-2,
\ 2\delta(V_r^k)-2,
\ h^{\max}(V_r^{k-1}),
\ h^{\max}(V_r^{k}),
\ h^{\max}(V_r^{k+1}) ).\]
It follows by induction  for all $r \geq d$ that $\delta(V_r^k)\leq D_k$ and \[h^{\max}(V_r^k)\leq 
\max\big(\,\max_{\ell \leq k+r-d} \eta_\ell,\ \ 
\max_{\ell \le k+r-d-1} (2D_\ell-2)\,\big).\]
Since $V_\infty^k=V_{\max(k+2, d)}^k$, we find that 
$\delta(V_\infty^k)\leq D_k$ and 
\[h^{\max}(V_\infty^k)\leq 
\max\big(\,\max_{\ell \leq k+s-d} \eta_\ell,\ \ 
\max_{\ell \le 2k-d+1} (2D_\ell-2)\,\big),\]
as desired.\end{proof}

Recall that if an $\FI$-module $V$ is semi-induced then $h^{\max}(V) = -1$. Hence we obtain the following corollary by using Proposition~\ref{prop:non-quillen} to bound $\delta(M^k)$ and $h^{\max}(M^k)$, then applying Proposition~\ref{prop:tot-degree1}\ref{bound:t0} and \ref{bound:t1}.
\begin{corollary}
\label{cor:FIsharpTypeA}
Let $E_r^{p,q}$ be a cohomologically graded first quadrant spectral sequence of $\FI$-modules converging to $M^{p+q}$. Suppose that for some page $d$, the $\FI$-modules $E_d^{p,q}$ are semi-induced and generated in degree $\leq \mu(p+q)$ for some $\mu$. Then we have \begin{enumerate}
\item $\delta(M^k) \le \mu k$.
\item $h^{\max}(M^k) \le \max(-1, 4\mu k-2\mu(d-1)-2)$.
\item $t_0(M^k) \leq \max(\mu k,  5 \mu k-2\mu(d-1) - 1)$.
\item $t_1(M^k) \leq \max(\mu k,  9\mu k - 4 \mu(d-1) - 2)$.
\end{enumerate} 
\end{corollary}

\subsection{Cohomology of configuration spaces}
Let $\cM$ be a connected manifold of dimension $d \geq 2$. Let $A$ be any abelian group. In this section, we prove a linear bound on the generation and presentation degrees of the $\FI$-modules $\rH^k(\PConf(\cM); A)$ described in \S \ref{sec:intro}. The following theorem (in conjunction with Proposition~\ref{prop:equalizer} and Proposition~\ref{prop:over-fields}) includes Application \ref{ap:config} as a special case.

\begin{theorem} \label{thm:config}
Let $\cM$ be a connected manifold of dimension  $d \ge 2$, and set: \begin{align*}
\mu= \begin{cases} 2 &\mbox{if } d = 2 \\
1 &\mbox{if } d \geq 3
\end{cases}\qquad\qquad\lambda = \begin{cases} 0 &\mbox{if $M$ is non-orientable} \\
1 &\mbox{if $M$ is orientable}
\end{cases}
\end{align*} 

Let $A$ be an abelian group. Then we have: \begin{enumerate}
\item $\delta(\rH^k(\PConf(\cM);A)) \le \mu k$.
\item $h^{\max}(\rH^k(\PConf(\cM);A)) \le \max( -1, 4 \mu k-2\mu\lambda-2)$.
\item $t_0(\rH^k(\PConf(\cM);A)) \leq \max(\mu k, 5 \mu k-2\mu\lambda- 1)$.
\item $t_1(\rH^k(\PConf(\cM);A)) \leq max(\mu k, 9\mu k - 4 \mu\lambda - 2)$.
\end{enumerate} 
\end{theorem}

This follows immediately from Corollary~\ref{cor:FIsharpTypeA}, in light of the following two results.
The first is due to Miller and Wilson; it follows from the proof of \cite[Theorem A.12]{MW}. The second is due to Totaro \cite{Tot}. 
\begin{theorem}[Miller--Wilson] \label{thm:MWss}
There is a first quadrant spectral sequence $E^{p,q}_r$ of $\FI$-modules converging to $\rH^{p+q}(\PConf(\cM), A)$  such that $E^{p,q}_1$ is induced and $t_0(E^{p,q}_1) \leq \mu q$. 
\end{theorem}

\begin{theorem}[{Totaro \cite{Tot}, see \cite[Proof of Theorem~6.3.1]{fimodules}}] \label{thm:Totaross}
If $\cM$ is orientable, there is a first quadrant spectral sequence $E^{p,q}_r$ of $\FI$-modules converging to $\rH^{p+q}(\PConf(\cM); A)$  such that $E^{p,q}_2$ is induced and $t_0(E^{p,q}_2) \leq \mu (p+q)$. 
\end{theorem}

One can improve Theorem~\ref{thm:config} if the manifold admits two pointwise linearly independent vector fields. This includes all manifolds with trivial tangent bundle.  We give this example because it illustrates that sometimes one can bound $h^{\max}$ using topology instead of algebra.

\begin{proposition} \label{prop:vectorfield}
With the notation of Theorem~\ref{thm:config}, suppose that $\cM$  admits a pair of linearly independent vector fields.  Then we have: \begin{enumerate}[label=(\arabic*)]
\item\label{part:2vecD} $\delta(\rH^k(\PConf(\cM);A)) \le \mu k$.
\item\label{part:2vecH} $h^{\max}(\rH^k(\PConf(\cM);A)) \le 0$.
\item\label{part:2vect0} $t_0(\rH^k(\PConf(\cM);A)) \leq \mu k + 1$.
\item\label{part:2vect1} $t_1(\rH^k(\PConf(\cM);A)) \leq \mu k + 2$.
\end{enumerate} 
\end{proposition}

\begin{proof} We will need the following three categories:
 \begin{itemize}
\item Let $\Sharp$ denote the category with objects finite based sets and with morphisms given by maps of based sets such that the preimage of all elements except possibly the base point have cardinality at most one. 
\item Let $\Set$ denote the category of finite sets and all maps.
\item Let $\Set_{\star}$ denote the category of based sets and base point-preserving maps.
\end{itemize}  There is a commuting square of natural functors: 
$$
\xymatrix{
\FI  \ar[d] \ar[r] & \Sharp \ar[d] \\ 
\Set \ar[r] & \Set_{\star}
}
$$
Church--Ellenberg--Farb \cite[Theorem 4.1.5]{fimodules} proved that the restriction of an $\Sharp$-module to $\FI$ is an induced $\FI$-module. Thus, $\Set_{\star}$-modules are also induced $\FI$-modules. Ellenberg and Wiltshire-Gordon \cite[Theorem 14]{EWG} showed that  the $\FI$-module structure on $\rH^k(\PConf(\cM); A)$ extends to the structure of a $\Set$-module if $\cM$ admits a pair of linearly independent vector fields. The shift of a $\Set$-module is naturally a $\Set_{\star}$-module. We conclude that the $\FI$-module structure on $\Sigma \left( \rH^k(\PConf(\cM);A) \right)$ extends to a $\Set_{\star}$-module structure. Thus $\Sigma\left( \rH^k(\PConf(\cM);A) \right)$ is induced. By Corollary~\ref{cor:hmax}, $h^{\max}( \rH^k(\PConf(\cM);A))$ is at most 0, proving Part~\ref{part:2vecH}. Theorem~\ref{thm:config} gives Part~\ref{part:2vecD}, and  Proposition~\ref{prop:tot-degree1} then implies Parts~\ref{part:2vect0} and~\ref{part:2vect1}.
\end{proof}

\section{\TypeB{} spectral sequence arguments and congruence subgroups}
\label{secB}

In this section, we prove Theorem \ref{athm:growth-bounds} which establishes quadratic stable ranges in \TypeB{} spectral sequence arguments. We use this to prove our results on congruence subgroups of general linear groups, Application \ref{athm:congruence-subgroup}.

\subsection{The \TypeB{} setup} 
By a \TypeB{} setup, we mean that we start with a bounded-below complex $M_\bullet$ of FI-modules, together with bounds on  $\bt_k(M_{\bullet})$ for each $i$ (which typically grow linearly in $k$). There is a hyper-homology spectral sequence with the second page $E^2_{i,j} = \rH^{\FI}_i(\rH_j(M))$ converging to $\HH_{i+j}(M_{\bullet})$, and one can analyze this spectral sequence to produce bounds  on $t_0(\rH_j(M_{\bullet}))$ and $t_1(\rH_j(M_{\bullet}))$. Previous methods lead to bounds that are exponential in $j$ even if the bound on $\bt_k(M_{\bullet})$ is linear in $k$; see the proof of  \cite[Theorem~D]{castelnuovo-regularity}.  Our next theorem together with Proposition~\ref{prop:tot-degree1} provides a way to get better bounds in a \TypeB{} setup. It gives a quantitative version of Theorem \ref{athm:growth-bounds}.

\begin{theorem} 
\label{thm:growth-bounds}
Suppose $M_{\bullet}$ is a bounded-below complex of $\FI$-modules with $\bt_k(M_\bullet)<\infty$. Then for all $k$: \begin{enumerate}[label=(\arabic*)]
\item\label{part:deltaLEt} $\delta(\rH_k(M_{\bullet})) \le \bt_k(M_{\bullet})$.
\item\label{part:hmaxbound} $h^{\max}(\rH_k(M_{\bullet})) \le \max_{q<k} h^{\max}(\rH_q(M_{\bullet})) +  \max(\bt_k(M_{\bullet}), \bt_{k+1}(M_{\bullet}))  + \bt_k(M_{\bullet})$. 
\end{enumerate} 

\noindent  In particular, if $\bt_k(M_{\bullet}) \leq ak+b$ for all $k$ and $M_{\bullet}$ is supported on non-negative homological degrees, then we have:\begin{enumerate}[label =(\alph*)]
	\item\label{part:appBa} $\delta(\rH_k(M_{\bullet})) \leq ak+b$.
	\item\label{part:appBb} $h^{\max}(\rH_k(M_{\bullet})) \leq a (k+1)^2 + 2b(k+1) = a k^2 + 2(a+b) k + a + 2 b$.
	\item\label{part:appBc} $t_0(\rH_k(M_{\bullet})) \leq a k^2 + (3a+2b) k + a + 3 b + 1$.
	\item\label{part:appBd} $t_1(\rH_k(M_{\bullet})) \leq 2 a k^2 + (5a + 4b)k + 2a + 5b + 2$.
\end{enumerate}
\end{theorem}


\begin{remark}
In the initial distributed version of this paper, we asked if the above theorem could be improved to produce linear ranges. An affirmative answer to that question was subsequently found by Gan and Li \cite[Theorem~5]{gan-li}.
\end{remark}

The first step in proving the above theorem is to show that each $\rH_k(M_\bullet)$ is presented in finite degrees. This is accomplished by the following lemma which we have extracted from the proof of \cite[Theorem~D]{castelnuovo-regularity}.

\begin{lemma}
	\label{lemma:quillen-qualitative}
	Suppose $M_{\bullet}$ is a bounded-below complex of $\FI$-modules with $\bt_k(M_\bullet)<\infty$ for each $k$. Then each $\rH_k(M_\bullet)$ is presented in finite degrees.
\end{lemma}
\begin{proof} We assume, without loss of generality, that $M_{\bullet}$ is supported on non-negative homological degrees, and proceed by induction on $k$. We have the hyper-homology spectral sequence  with second page $E^2_{p,q} = \rH^{\FI}_p(\rH_q(M_{\bullet}))$ converging to $\HH_{p+q}( M_{\bullet})$. The following inequalities follow easily from it: \begin{align*}
	t_0(\rH_k(M_{\bullet})) \le \max( \bt_k(M_\bullet),  \max_{p+q = k+1, q< k} t_p(\rH_q(M_{\bullet}))) \\
	t_1(\rH_k(M_{\bullet})) \le \max( \bt_{k+1}(M_\bullet),  \max_{p+q = k+2, q < k} t_p(\rH_q(M_{\bullet}))) 
	\end{align*}
	In the base case $k=0$, we get $t_0(\rH_0(M_{\bullet})), t_1(\rH_0(M_{\bullet})) < \infty$. This is equivalent to $\rH_0(M_{\bullet})$ being presented in finite degrees. By induction,  assume that $\rH_q(M_\bullet)$ are presented in finite degrees for $q<k$. By Theorem~\ref{thm:homology-acyclics}\ref{part:abeliancategory}, we see that $t_p(\rH_q(M_{\bullet})) < \infty$ for $q < k$ and $p \ge 0$. The two inequalities above thus imply that  $\rH_k(M_{\bullet})$ is presented in finite degrees, completing the inductive step of the proof.
\end{proof}
\begin{proof}[Proof of Theorem~\ref{thm:growth-bounds}]
We assume, without loss of generality, that $M_{\bullet}$ is supported on non-negative homological degrees. By Lemma~\ref{lemma:quillen-qualitative}, each $\rH_k(M_\bullet)$ is presented in finite degrees, and so methods of the earlier sections (for example Theorem~\ref{thm:shift-theorem}) are applicable.

Proof of Part~\ref{part:deltaLEt}: Fix a $k$ and choose an $n$ large enough such that $\Sigma^n \rH_q(M_{\bullet})$ are semi-induced for $q \le k$. We have a spectral sequence  with $E^2_{p,q} = \rH^{\FI}_p(\rH_q(\Sigma^n M_{\bullet}))$ converging to $\HH_{p+q}(\Sigma^n M_{\bullet})$. Since semi-induced modules are $\FI$-homology acyclic (Theorem~\ref{thm:homology-acyclics}), we have $E^2_{p,q} = 0$ for $p > 0 $ and $q \le k$. This causes the spectral sequence to collapse in a range. We conclude that \[\deg E^2_{0, k} = \deg E^\infty_{0, k}= \delta(\rH_k(M_{\bullet})) \le \bt_k(\Sigma^n M_{\bullet}) \le \bt_k(M_{\bullet})\] where the last inequality is by Theorem~\ref{thm:shift-of-homology}. This proves Part~\ref{part:deltaLEt}.

Proof of Part~\ref{part:hmaxbound}: Fix a $k$ and choose an $n$ large such that $\Sigma^n \rH_q(M_{\bullet})$ are semi-induced for $q < k$. This $n$ is precisely $\max_{q<k} h^{\max}(\rH_q(M_{\bullet}))+1$ (Corollary~\ref{cor:hmax}). We have a spectral sequence  with $E^2_{p,q} = \rH^{\FI}_p(\rH_q(\Sigma^n M_{\bullet}))$ converging to $\HH_{p+q}(\Sigma^n M_{\bullet})$. Since semi-induced modules are $\FI$-homology acyclic, we have $E^2_{p,q} = 0$ for $p > 0 $ and $q < k$. This causes the spectral sequence to collapse in a range. This degeneration together with Theorem~\ref{thm:shift-of-homology} leads to the following inequalities: \begin{align*}
t_0( \rH_k(\Sigma^n M_{\bullet})) = \deg E^2_{0, k} = \deg E^\infty_{0, k} \le \bt_k (\Sigma^n M_{\bullet}) \le   \bt_k( M_{\bullet}) \\ 
t_1( \rH_k(\Sigma^n M_{\bullet})) = \deg E^2_{1, k} = \deg E^\infty_{1, k} \le \bt_{k+1}(\Sigma^n M_{\bullet}) \le \bt_{k+1}(M_{\bullet}).
\end{align*}

Since local cohomologies commute with shifts, we have $h^{\max}(\rH_k(M_{\bullet})) \le n + h^{\max}(\Sigma^n \rH_k( M_{\bullet}))$. Clearly, we also have  $\Sigma^n \rH_k(M_{\bullet}) =  \rH_k(\Sigma^n M_{\bullet})$. Now Theorem~\ref{thm:local-cohomology} together with Part~\ref{part:deltaLEt} and the inequalities above yield \[h^{\max}(\rH_k(\Sigma^n M_{\bullet})) \le \max(\bt_k( M_{\bullet}), \bt_{k+1}(M_{\bullet}))  + \bt_k(M_{\bullet}) - 1.\] Thus we have
\begin{align*}
h^{\max}(\rH_k(M_{\bullet})) & \le n + h^{\max}(\rH_k(\Sigma^n M_{\bullet}))\\
& \le n + \max(\bt_k( M_{\bullet}), \bt_{k+1}(M_{\bullet}))  + \bt_k(M_{\bullet}) - 1 \\
& = \max_{q<k} h^{\max}(\rH_q(M_{\bullet})) +  \max(\bt_k (M_{\bullet}), \bt_{k+1}(M_{\bullet}))  + \bt_k(M_{\bullet}).
\end{align*} This completes the proof of \ref{part:hmaxbound}. 

Now we prove the remaining parts. Part~\ref{part:appBa} follows from Proposition~\ref{prop:delta}\ref{part:deltat0gen}. For part \ref{part:appBb}, denote $\max_{q\le k} h^{\max}(\rH_q(M_{\bullet}))$ by $T_k$ and set $T_{-1}=0$. From part \ref{part:hmaxbound}, we have $T_k - T_{k-1} \le a(2k+1) +2b$. Expanding a telescoping sum gives \[T_k = T_k - T_{-1} \le \sum_{i = 0}^k a(2i+1) + 2b = a(k+1)^2 + 2b(k+1) = ak^2 + 2(a+b) k +a + 2b,\] completing the proof of \ref{part:appBb}. Parts \ref{part:appBc} and \ref{part:appBd} follow from Proposition~\ref{prop:tot-degree1} together with \ref{part:appBa} and \ref{part:appBb}.
\end{proof}

\subsection{Congruence subgroups}
\label{subsec:Application-B}

We now prove Application \ref{athm:congruence-subgroup}, stability for congruence subgroups. We will need the following proposition bounding the $\FI$-homology of the chains on congruence $\FI$-groups; the proof of this proposition occupies the next two sections (see \S \ref{sec:hyperhomologyvanishescong}).

Let $E_\bullet G$ denote a functorial resolution of $\Z$ by free $\Z G$-modules (e.g.\ the bar resolution). Then for any $\FI$-group $\Gamma$, we may consider $C_\bullet\Gamma\coloneq E_\bullet\Gamma\otimes_\Gamma \Z$, which is a chain complex of $\FI$-modules with $\rH_k(C_\bullet \Gamma)\iso \rH_k(\Gamma;\Z)$.
 For any  coefficient group $A$ we set $C_\bullet(\Gamma;A)=C_\bullet(\Gamma)\otimes A$.

\begin{proposition} \label{hyperhomologyvanishescong}
Let $I$ be a proper ideal in a ring $R$ satisfying Bass's stable range condition $\SR_{d+2}$, and let $\Gamma=\GL(R,I)$ be the corresponding congruence $\FI$-group. For any coefficient group $A$ we have $\bt_k(C_\bullet (\Gamma;A)) \leq  2k+d$.
\end{proposition}

We now prove Application \ref{athm:congruence-subgroup}, stability for congruence subgroups. 
\begin{proof}[Proof of Application \ref{athm:congruence-subgroup}]		
Theorem \ref{thm:growth-bounds} and Proposition \ref{hyperhomologyvanishescong} give Parts \ref{b1}, \ref{b2}, \ref{b3}, and \ref{b4} of Application \ref{athm:congruence-subgroup} in the case that the ideal is proper. We invoke Proposition~\ref{prop:equalizer} and Proposition~\ref{prop:over-fields} to deduce Parts \ref{polycong} and \ref{polycoeq} of Application \ref{athm:congruence-subgroup}.

For $I=R$, much stronger results are already known, since the groups $\GL_n(R,I) \cong\GL_n(R)$ exhibit classical homological stability. Van der Kallen \cite{vdK} showed that $\rH_k(\GL_n(R)) \m \rH_k(\GL_{n+1}(R))$ is an isomorphism for $n \geq 2k+\max(1,d)$ and a surjection for $n \geq 2k+\max(1,d) -1$. This shows that $\rH_k(\GL(R))$ has stable degree $\leq 0$, generation degree $\leq 2k+\max(1,d) -1$, local degree $\leq 2k+\max(1,d) -1$ (see Theorem~\ref{thm:local-cohomology}(3)) and presentation degree $\leq 2k+\max(1,d)$. These bounds are all at least as good as Application \ref{athm:congruence-subgroup} claims, so Application \ref{athm:congruence-subgroup} is also true for $I=R$. 
\end{proof}

\section{The $\FI$-homology of the chains on an $\FI$-group} \label{secFIgroup}

For any injective $\FI$-group $\Gamma$ (meaning that all maps $\Gamma_T\to \Gamma_R$ are injective) Church--Ellenberg defined in \cite[Proposition~5.13]{castelnuovo-regularity}
a natural complex of $\FI$-modules $X_\Gamma$ on which $\Gamma$ acts.\footnote{The hypothesis that $\Gamma$ is injective was not mentioned explicitly in \cite{castelnuovo-regularity}, but the definition of $X_\Gamma$ only makes sense when $\Gamma$ is injective, as we will see below.}
 The definition of $X_\Gamma$ will be spelled out explicitly below in Definition~\ref{def:XGamma}.

The purpose of this section is to prove the following proposition:
\begin{proposition}
\label{prop:homologyofchains}
For any injective $\FI$-group $\Gamma$, there is a quasi-isomorphism \[\rL \rH^{\FI}(C_\bullet \Gamma)=\rL \rH^{\FI}(E_\bullet \Gamma\otimes_\Gamma \Z)\to E_\bullet\Gamma \otimes_\Gamma  X_\Gamma.\] In particular, $\HH_k(C_\bullet \Gamma)\iso \rH_k(\Gamma;X_\Gamma)$.
\end{proposition}
In other words, the $\FI$-homology of the chains on the $\FI$-group $\Gamma$ is computed by the $\Gamma$-equivariant homology of the $\FI$-complex $X_\Gamma$.
Since $X_\Gamma$ consists of free abelian groups, Proposition~\ref{prop:homologyofchains} has the following corollary.
\begin{corollary}
\label{cor:homologyofchainscoeff}
$\HH_k(C_\bullet (\Gamma;A))\iso\rH_k(\Gamma;X_\Gamma\otimes A)$ for any coefficient group $A$.
\end{corollary}
In light of these results, one immediately sees the importance of  understanding the $\FI$-complex  $X_\Gamma$. In service to this aim, we give in Section~\ref{sec:describecomplex} an explicit description for two representative $\FI$-groups: the general linear group $\GL(R)$ and its congruence subgroup $\GL(R,I)$.

\subsection{The complex of $\FI$-modules $V(F)$} Throughout the proof of Proposition~\ref{prop:homologyofchains}, we will make use of the following construction (originating in \cite[Construction~5.7]{castelnuovo-regularity}).

Let $\FIarrow$ denote the arrow category of $\FI$, which we will think of as follows: an object of $\FIarrow$ is a pair $(T,S)$ where $T$ is a finite set and $S\subset T$ is a subset (not necessarily proper). We will denote such an object as $S\subset T$, but we emphasize that this is merely formal notation for the pair $(T,S)$. A morphism from $S\subset T$ to $S'\subset T'$ is an injection $f\colon T\into T'$ such that $f(S)\subset S'$.

For any finite set $R$, we denote by $\orient_R$ the space of \emph{orientations} of $R$. As an abelian group $\orient_R$ is always isomorphic to $\Z$. A bijection $R\xrightarrow{\iso}R'$ induces an isomorphism $\orient_R\xrightarrow{\iso}\orient_{R'}$, and permutations of $R$ act on $\orient_R$ by the sign representation. When we wish to emphasize that $R$ has cardinality $\abs{R}=k$, we write $\orient^k_R$.

\begin{definition}
We define an exact functor $V\colon \FIarrowMod\to \Ch(\FIMod)$ as follows. On objects, $V(F)_T$ is the graded abelian group 
\[V(F)_T\coloneq \bigoplus_{T=S\disjoint R}F_{S\subset T}\boxtimes \orient_R,\] graded by the cardinality of $R$. That is,
\begin{equation}
\label{eq:VFTk}
(V(F)_T)_k= \bigoplus_{\substack{T=S\disjoint R\\\abs{R}=k}}F_{S\subset T}\boxtimes \orient^k_R.
\end{equation}
(The symbol $\boxtimes$ here means simply $\otimes_\Z$; we write $\boxtimes$ for visual distinctness, and to emphasize later that certain constructions act only on the first factor of $F_{S\subset T}\boxtimes \orient_R$.)

An $\FI$-morphism $f\colon T\into T'$ sends the summand indexed by $T=S\disjoint R$ to the summand indexed by $T'=S'\disjoint f(R)$ for $S'\coloneq T'\setminus f(R)$. If $\abs{R}=k$, the map \[F_{S\subset T}\boxtimes \orient^k_R\to F_{S'\subset T'}\boxtimes \orient^k_{f(R)}\] is given on the first factor  by  the $\FIarrow$-morphism $f_*\colon F_{S\subset T}\to F_{S'\subset T'}$ and on the second factor by the $\FI$-morphism $f_*\colon \orient^k_R\xrightarrow{\iso}\orient^k_{f(R)}$. This gives $V(F)$ the structure of a graded $\FI$-module.


Finally, the differential $d$ is given on the summand $F_{S\subset T}\boxtimes \orient^k_R$ by the alternating sum over $r\in R$ of maps \[F_{S\subset T}\boxtimes \orient^k_R\to F_{S\disjoint\{r\}\subset T}\boxtimes \orient^{k-1}_{R\setminus\{r\}}.\]
That $V$ is exact is visible from the defining formula \eqref{eq:VFTk}. 
\end{definition}
\begin{remark}
This construction is due to Church--Ellenberg, and was denoted $C^F_\bullet$ in \cite[Construction~5.7]{castelnuovo-regularity}. In particular, the reader concerned about consistency of signs or checking that $d^2=0$ can find all details spelled out in \cite{castelnuovo-regularity}. However be aware that this construction was phrased there in terms of the \emph{twisted} arrow category $\FItwist$; our description here corresponds to \cite[Construction~5.7]{castelnuovo-regularity} under the isomorphism $\FIarrow\iso \FItwist$ given by $S\subset T\ \ \mapsto\ \ (T\setminus S)\subset T$.
\end{remark}

The following is a standard result on triangulated categories.

\begin{lemma}
	Let $\cA, \cA'$ be abelian categories, and let $X \colon \Ch^+(\Mod_{\cA} ) \to \Ch^+(\Mod_{\cA'} )$ be a functor. Then $X$ induces a triangulated functor  $X \colon \rD^+(\Mod_{\cA} ) \to \rD^+(\Mod_{\cA'} )$ if the following hold: \begin{enumerate}
		\item $X$ commutes with homological shifts.
		\item $X$ maps distinguished triangles $M \xrightarrow{f} N \to \Cone(f)$ to distinguished triangles.
		\item $X$ takes quasi-isomorphisms to quasi-isomorphisms. 
	\end{enumerate}
\end{lemma}
%
%
%
%
%
%
%

\begin{lemma}
\label{lem:construction}
$V$ induces  a triangulated functor $V\colon \rD^+(\FIarrowMod)\to \rD^+(\FIMod)$.
\end{lemma}

\begin{proof}
If $F=F_\bullet$ is a bounded-below chain complex of $\FIarrow$-modules, the functoriality of $V$ makes $V(F_q)_p$ into a double complex of $\FI$-modules; we define $V(F)$ to be its total complex, satisfying 
\[(V(F)_T)_k=\bigoplus_{\substack{p+q=k\\}}\bigoplus_{\substack{T=S\disjoint R\\\abs{R}=p}}(F_q)_{S\subset T}\boxtimes \orient^p_R.\]
This defines an extension $V\colon \Ch^+(\FIarrowMod)\to \Ch^+(\FIMod)$. It now suffices to check the conditions from the previous lemma, of which the first two are trivial. For the third condition, note that from the obvious vertical grading we have a spectral sequence \[(E^0_{p\bullet})^F=\bigoplus_{\substack{T=S\disjoint R\\\abs{R}=p}}(F_\bullet)_{S\subset T}\boxtimes \orient^p_R\quad\implies\quad \rH_k(V(F)).\] If $f\colon F_\bullet\to G_\bullet$ is a quasi-isomorphism, the exactness of the original $V$ means the induced map $(E^0_{p\bullet})^F\to (E^0_{p\bullet})^G$ will be a quasi-isomorphism with respect to $d^0$. It follows that $f$ induces an isomorphism on $E^1$, so $f\colon V(F)\to V(G)$ is a quasi-isomorphism. This shows that $V$ descends to a triangulated functor $V\colon \rD^+(\FIarrowMod)\to \rD^+(\FIMod)$, as claimed.
\end{proof}



%
%
%
%

There are \emph{two} natural ways to view an $\FI$-module as an $\FIarrow$-module, and both will be used in this section. They arise from the two natural functors $\FIarrow\to \FI$ defined by
\[(T,S)\mapsto S\qquad\text{and}\qquad(T,S)\mapsto T\]
Accordingly, if $M$ is an $\FI$-module, let  $\innerpull{M}$ and $\outerpull{M}$ be the $\FIarrow$-modules defined by
\[\innerpull{M}_{S\subset T}=M_S\qquad\text{and}\qquad\outerpull{M}_{S\subset T}=M_T.\]
Both  $M\mapsto\innerpull{M}$ and $M\mapsto \outerpull{M}$ are exact, so we use the same notation for $M\in \rD^+(\FIMod)$.
\begin{lemma}
\label{lemma:VMbar-homology}
For any $M\in \rD^+(\FIMod)$ we have $V(\innerpull{M})\iso\rL \rH^{\FI}(M)$.
\end{lemma}
\begin{proof}
It follows from \cite[Proof of Thm C]{castelnuovo-regularity} and \cite[Proposition~5.10]{castelnuovo-regularity}.
\end{proof}

Given an $\FI$-group $\Gamma$, as above we write $\outerpull{\Gamma}$ for the $\FIarrow$-group defined by $\outerpull{\Gamma}_{S\subset T}=\Gamma_T$.
\begin{lemma}
\label{lemma:VAF}
Suppose that $F$ is an $\FIarrow$-module with a left action of $\outerpull{\Gamma}$, and $A$ is an $\FI$-module (or chain complex of $\FI$-modules) with a right action of $\Gamma$. 
\begin{enumerate}
\item The action of $\outerpull{\Gamma}$ on $F$ induces an action of $\Gamma$ on the $\FI$-module $V(F)$.
\item There is a natural isomorphism $V(\outerpull{A}\otimes_{\outerpull{\Gamma}}F)\iso A\otimes_\Gamma V(F)$ as  complexes of $\FI$-modules.
\end{enumerate}
\end{lemma}
\begin{proof}
 Let us first spell out what it means for $\outerpull{\Gamma}$ to act on the $\FIarrow$-module $F$. This means for every $T$ and $S\subset T$ we have an action of the group $\Gamma_T$ on the $\Z$-module $F_{S\subset T}$; and for any $f\colon T\to T'$ and any $S'\supseteq f(S)$ the $\FIarrow$-morphism $f_*\colon F_{S\subset T}\to F_{S'\subset T'}$ is equivariant with respect to the $\FI$-morphism $f_*\colon \Gamma_T\to \Gamma_{T'}$.
 
In particular, the $\FIarrow$-module $\outerpull{A}\otimes_{\outerpull{\Gamma}}F$ is defined by \[(\outerpull{A}\otimes_{\outerpull{\Gamma}}F)_{S\subset T}=\outerpull{A}_{S\subset T}\otimes_{\Gamma_{S\subset T}}F_{S\subset T}=A_T\otimes_{\Gamma_T}F_{S\subset T}.\]

(1): The claimed action of $\Gamma_T$  on $V(F)_T=\bigoplus_{T=S\disjoint R}F_{S\subset T}\boxtimes \orient_R$
preserves each summand; on the summand $F_{S\subset T}\boxtimes \orient_R$ it acts on $F_{S\subset T}$ by the specified action of $\Gamma_{T}=\outerpull{\Gamma}_{S\subset T}$, and acts on $\orient_R$ by the identity. It is straightforward to check that this commutes with $\FI$-morphisms and the differential.


(2): We have 
\begin{align*}
(A\otimes_\Gamma V(F))_T&=A_T\otimes_{\Gamma_T}\Big(\bigoplus_{T=S\disjoint R} F_{S\subset T}\boxtimes \orient_R\Big)\\
&\iso\bigoplus_{T=S\disjoint R} A_T\otimes_{\Gamma_T}\big(F_{S\subset T}\boxtimes \orient_R\big)\\
&\iso\bigoplus_{T=S\disjoint R}(A_T\otimes_{\Gamma_T}F_{S\subset T})\boxtimes \orient_R\qquad(\ast)\\
&\iso\bigoplus_{T=S\disjoint R}(\outerpull{A}\otimes_{\outerpull{\Gamma}} F)_{S\subset T}\boxtimes \orient_R\qquad\\
&=V(\outerpull{A}\otimes_{\outerpull{\Gamma}} F)_T
\end{align*}
The key isomorphism is $(\ast)$; this holds because the action of $\Gamma_T$ on the summand $F_{T,U}\boxtimes \orient_R$ of $V(F)$ is trivial on the $\orient_R$ factor. To see that this isomorphism commutes with the differential, recall that the differential of $V(\outerpull{A}\otimes_{\outerpull{\Gamma}} F)$ is the alternating sum of maps induced by $(\outerpull{A}\otimes_{\outerpull{\Gamma}} F)_{S\subset T}\to (\outerpull{A}\otimes_{\outerpull{\Gamma}} F)_{S\disjoint\{r\}\subset T}$. But since $\outerpull{A}$ is pulled back from an $\FI$-module, $\FIarrow$-morphisms of the form $\id\colon (S\subset T)\to (S\disjoint\{r\}\subset T)$ act by the \emph{identity} on $\outerpull{A}$. So the differential is the identity on the $A_T$ factor of $(\outerpull{A}\otimes_{\outerpull{\Gamma}} F)_{S\subset T}\iso A_T\otimes_{\Gamma_T}F_{S\subset T}$, just as we need.
The functoriality of $(\ast)$ guarantees that this holds for a chain complex $A$ of $\FI$-modules as well.
\end{proof}

\subsection{The complex $X_\Gamma$ and the proof of Proposition~\ref{prop:homologyofchains}}
We can now define the complex $X_\Gamma$. Suppose that $\Gamma$ is an injective $\FI$-group. For any $S\subset T$, we can identify $\Gamma_S$ with the subgroup $i_*(\Gamma_S)\subset \Gamma_T$, where $i\colon S\into T$ is the inclusion. We will denote this subgroup simply by $\Gamma_S\subset \Gamma_T$.

\begin{definition}
\label{def:XGamma}
We define the $\FIarrow$-module $F_\Gamma$ as follows. On objects, $(F_\Gamma)_{S\subset T}\coloneq \Z[\Gamma_T/\Gamma_S]$. Given a morphism $f\colon T\to T'$ with $f(S)\subset S'$, we define $f_*\colon (F_{\Gamma})_{S\subset T}\to (F_{\Gamma})_{S'\subset T'}$ to be the map induced by $f_*\colon \Z[\Gamma_T]\to \Z[\Gamma_{T'}]$; since $f(S)\subset S'$, we have $f(\Gamma_S)\subset \Gamma_{S'}$, so this is well-defined.

We define the complex $X_\Gamma$ of $\FI$-modules to be $X_\Gamma\coloneq V(F_{\Gamma})$.
The obvious left action of $\Gamma_T$ on $(F_{\Gamma})_{S\subset T}$ induces an action of $\outerpull{\Gamma}$ on $F_{\Gamma}$, so by Lemma~\ref{lemma:VAF}(1) it induces an action of $\Gamma$ on the chain complex of $\FI$-modules $X_\Gamma=V(F_{\Gamma})$.

Explicitly, the complex $X_\Gamma$ is given in homological degree $k$ by
\begin{equation}
\label{eq:XGammadef}
((X_\Gamma)_T)_k=\bigoplus_{\substack{T=S\disjoint R\\\abs{R}=k}}\Z[\Gamma_T/\Gamma_S]\boxtimes \orient^k_R,
\end{equation}
with $\Gamma_T$ preserving each summand, and acting only on $\Z[\Gamma_T/\Gamma_S]$.
\end{definition}

\begin{proof}[Proof of Proposition~\ref{prop:homologyofchains}]
First, let us show that there is a natural quasi-isomorphism between the complexes of $\FIarrow$-modules 
\begin{equation}
\label{eq:qiCGammaEGamma}
\innerpull{C_\bullet \Gamma}\xrightarrow{\iso} \outerpull{E_\bullet \Gamma}\otimes_{\outerpull{\Gamma}} F_\Gamma.
\end{equation}
Indeed, on objects we have by definition
\begin{align*}
(\innerpull{C_\bullet \Gamma})_{S\subset T}=C_\bullet \Gamma_S&=E_\bullet \Gamma_S\otimes_{\Gamma_S} \Z\\(\outerpull{E_\bullet \Gamma}\otimes_{\outerpull{\Gamma}} F_\Gamma)_{S\subset T}\phantom{\,=C_\bullet\Gamma_S}&=E_\bullet \Gamma_T\otimes_{\Gamma_T} \Z[\Gamma_T/\Gamma_S].
\end{align*}
The desired map from the former to the latter is induced by the inclusion $\Gamma_S\into \Gamma_T$, and is a quasi-isomorphism by Shapiro's lemma.

We now finish the proof:
\begin{align*}\rL \rH^{\FI}(C_\bullet \Gamma)
&\iso V(\innerpull{C_\bullet \Gamma})&&\text{by Lemma~\ref{lemma:VMbar-homology}} \\
&\iso V(\outerpull{E_\bullet \Gamma}\otimes_{\outerpull{\Gamma}} F_\Gamma)&&\text{by \eqref{eq:qiCGammaEGamma} and Lemma~\ref{lem:construction}}\\
&\iso E_\bullet \Gamma \otimes_\Gamma V(F_\Gamma)&&\text{by Lemma~\ref{lemma:VAF}(2)}\\
&=E_\bullet \Gamma \otimes_\Gamma X_\Gamma&&\text{by definition of $X_\Gamma$}
\end{align*}
This demonstrates that $\rL \rH^{\FI}(C_\bullet \Gamma)\iso E_\bullet \Gamma \otimes_\Gamma X_\Gamma$, as claimed.
\end{proof}

\subsection{Explicit description of $X_\Gamma$}
\label{subsec:explicit-description}
In this section we show that for most $\FI$-groups, the chain complex $X_\Gamma$ can be realized by a concrete $\FI$-simplicial complex.
\begin{definition}
\label{def:YGamma}
For any injective $\FI$-group $\Gamma$, we define an $\FI$-simplicial complex $Y_\Gamma$ with an action of $\Gamma$ as follows.
\begin{itemize}
\item The vertices of the simplicial complex $(Y_\Gamma)_T$ are pairs $(\{t\},\proj{\gamma}\in \Gamma_T/\Gamma_{T\setminus\{t\}})$. 
\item Every maximal simplex of $(Y_{\Gamma})_T$ is $(\abs{T}-1)$-dimensional. For every $\gamma\in \Gamma_T$, the $\abs{T}$ vertices $(\{t\},\pi_t(\gamma))$ form a $(\abs{T}-1)$-simplex, where $\pi_t$ is the canonical projection $\pi_t\colon \Gamma_T\twoheadrightarrow \Gamma_T/\Gamma_{T\setminus\{t\}}$.
\end{itemize}
To define the $\FI$-structure and $\Gamma$-action, it suffices to define them on vertices (and check that simplices are preserved).  We must define for each $\FI$-morphism $f\colon T\into T'$ a map $f_*$ from the vertices of $(Y_{\Gamma})_T$ to the vertices of $(Y_{\Gamma})_{T'}$, and for each $g\in \Gamma_T$ a map from the vertices of $(Y_{\Gamma})_T$ to itself. Given $\gamma\in \Gamma_T$, we define \begin{align*}
f_*(\{t\},\proj{\gamma})&=(\{f(t)\},\proj{f_*(\gamma)}\in \Gamma_{T'}/\Gamma_{T'\setminus \{t'\}})\\
g\cdot (\{t\},\proj{\gamma})&=(\phantom{f}\{t\}\phantom{()},\phantom{f}\proj{g\gamma}\phantom{|_*}\!\in\Gamma_{T}/\Gamma_{T\setminus \{t\}}\ \ \,)
\end{align*}
\end{definition}
Observe that $Y_\Gamma$ has a canonical dimension-preserving projection $Y_\Gamma\to\Delta^{\bullet-1}$ given by $(\{t\},\proj{\gamma})\mapsto \{t\}$. Here $\Delta^{\bullet-1}$ is the FI-simplicial complex which assigns to a set $T$ the complete simplex $\Delta^{T-1}$ on vertex set $T$. Furthermore, this projection realizes the quotient by the action of $\Gamma$, giving an identification $\Gamma\backslash Y_\Gamma\iso\Delta^{\bullet-1}$. This can be seen as follows.

Our definition of $(Y_\Gamma)_T$ implies that a collection $\sigma=\{(\{t_1\},\proj{\gamma_1}),\ldots,(\{t_k\},\proj{\gamma_k})\}$ forms a $(k-1)$-simplex if and only if the elements $t_1,\ldots,t_k\in T$ are all distinct, and there exists $\gamma\in \Gamma$ such that $\pi_{t_i}(\gamma)=\proj{\gamma_i}\in \Gamma_T/\Gamma_{T\setminus\{t_i\}}$ for all $i=1,\ldots,k$. This last condition can be rephrased as saying that there exists $\gamma\in \Gamma_T$ such that $\sigma=\gamma\cdot \{(\{t_1\},\proj{\id}),\ldots,(\{t_k\},\proj{\id})\}$. Therefore two simplices of $(Y_\Gamma)_T$ are in the same $\Gamma_T$-orbit if and only if they project to the same simplex $R=\{t_1,\ldots,t_k\}$ of $\Delta^{T-1}$, as claimed.


We now need the following condition on an $\FI$-group.
\begin{definition}
An injective $\FI$-group $\Gamma$ is \emph{saturated} if  for any $S_1\subset T$ and $S_2\subset T$, \[\Gamma_{S_1}\cap \Gamma_{S_2}=\Gamma_{S_1\cap S_2}\]
as subgroups of $\Gamma_T$.
\end{definition}
By induction, this implies that $\bigcap \Gamma_{S_i}=\Gamma_{\bigcap S_i}$ for any finite collection of subsets $S_i$.

\begin{remark} 
The same definition makes sense for $\FI$-modules (being abelian $\FI$-groups). An $\FI$-module $M$ is saturated if and only if \[\rH^0_{\fm}(M)=\rH^1_{\fm}(M)=0,\] as we now explain. By definition an $\FI$-module $M$ is injective as an $\FI$-group if and only if $M$ is torsion-free, i.e.\ if $\rH^0_{\fm}(M)=0$.
	
We next claim that under this assumption the condition $M_{S_1}\cap M_{S_2}=M_{S_1\cap S_2}$ holds if and only if the cokernel $\iterDelta_n M$ of the natural map $M \to \Sigma^n M$ is torsion-free for each $n \ge 0$. To see this equivalence, let $S_1, S_2$ and $T$ be sets such that $T = S_1 \cup S_2$ and $\abs{T\setminus S_1}=n$. Then the map $M \to \Sigma^n M$ in degrees $\abs{S_1 \cap S_2}$ and $\abs{S_1}$ leads to the commutative diagram  \begin{displaymath}
	\begin{tikzcd}
	M_{S_1}  \ar[hookrightarrow]{r} & (\Sigma^n M)_{S_1} \iso M_{T}  \ar{r}&(\iterDelta_n M)_{S_1}\ar{r}&0\\
	M_{S_1 \cap S_2}  \ar[hookrightarrow]{r} \ar[hookrightarrow]{u} &  (\Sigma^n M)_{S_1 \cap S_2} \iso M_{S_2} \ar{r}\ar[hookrightarrow]{u} &(\iterDelta_n M)_{S_1\cap S_2}\ar{r}\ar{u}&0
	\end{tikzcd}
	\end{displaymath} where the vertical maps are induced by the $\FI$-morphism $S_1 \cap S_2 \to S_1$ and horizontal maps are components of the natural map $M \to \Sigma^n M$. The vertical map $(\iterDelta_n M)_{S_1\cap S_2}\to (\iterDelta_n M)_{S_1}$ is thus injective if and only if $M_{S_1}\cap M_{S_2}=M_{S_1\cap S_2}$, as claimed.
	
By \cite[Proposition~1.1]{djament} and \cite[p.371, Corollaire]{gabriel}, this implies that $M$ is saturated if and only if  $\rH^0_{\fm}(M)=\rH^1_{\fm}(M)=0$. In particular, our definition of saturation agrees with the classical one; see \cite[III \S 2]{gabriel}. 
\end{remark}

It seems that ``natural'' examples of injective $\FI$-groups are almost always saturated; this includes $\GL(\Z)$, $\Aut(F_\bullet)$, and their congruence subgroups. An example of an injective $\FI$-group that is \emph{not} saturated would be the sub-$\FI$-group $\Gamma\subset \GL(\Z)$ defined by \[\Gamma_n=\begin{cases}\GL_n(\Z)&\text{if $n\geq 5$}\\1&\text{if $n<5$}\end{cases}\]
Any injective $\FI$-group $\Gamma$ is contained in a saturated $\FI$-group $\Gamma'$ defined by taking \[\Gamma'_T\coloneq \colim \left( \left(\Gamma_{T\disjoint [0]} \right)^{S_{[0]}} \m \left(\Gamma_{T\disjoint [1]}\right)^{S_{[1]}} \m \left(\Gamma_{T\disjoint [2]}\right)^{S_{[2]}} \m \ldots \right).\]

\begin{proposition}
\label{prop:XGammaYGamma}
If $\Gamma$ is a saturated $\FI$-group, then $X_\Gamma[1]$ is the reduced cellular chain complex of the $\FI$-simplicial complex $Y_\Gamma$.
\end{proposition}
\begin{proof}
The key to this proposition is that \emph{when $\Gamma$ is saturated}, the $\FI$-simplicial complex $Y_\Gamma$ admits the following alternative description:

The $(k-1)$-simplices of $(Y_{\Gamma})_T$ are in bijection with pairs $(R,\proj{\gamma}\in \Gamma_T/\Gamma_{T\setminus R})$ where $R\subset T$ and $\abs{R}=k$. The simplex $(R',\proj{\gamma'})$ is contained in the simplex $(R,\proj{\gamma})$ if and only if $R'\subset R$ and $\proj{\gamma'}$ is the image of $\proj{\gamma}$ under the projection $\Gamma_T/\Gamma_{T\setminus R}\twoheadrightarrow \Gamma_T/\Gamma_{T\setminus R'}$.

\medskip
To verify this description, consider a simplex $\sigma=\{(\{r_1\},\proj{\gamma_1}),\ldots,(\{r_k\},\proj{\gamma_k})\}$ of $Y_\Gamma$ lying above $R=\{r_1,\ldots,r_k\}\subset T$. By definition, for $\sigma$ to be a simplex means there exists $\gamma\in \Gamma_T$ with $\pi_{r_i}(\gamma)=\proj{\gamma_i}\in \Gamma_T/\Gamma_{T\setminus\{r_i\}}$ for all $i$. This element $\gamma$ is not unique; it is only well-defined modulo the intersection $\bigcap_{i=1}^k \Gamma_{T\setminus \{r_i\}}$. But the hypothesis that $\Gamma$ is saturated guarantees that \[\bigcap_{i=1}^k\Gamma_{T\setminus \{r_i\}}=\Gamma_{\bigcap_{i=1}^kT\setminus \{r_i\}}=\Gamma_{T\setminus R}.\] Therefore to every simplex $\sigma$ of $Y_\Gamma$ above $R\subset T$ determines an element $\proj{\gamma}\in \Gamma_T/\Gamma_{T\setminus R}$, and the containment relation is as described.

From this description the identification is clear: the $(k-1)$-simplices of $Y_\Gamma$ are labeled by
\begin{equation}
\label{eq:YGammasimplices}
((Y_\Gamma)_T)_{k-1}=\bigsqcup_{\substack{R\subset T\\\abs{R}=k}}\Gamma_T/\Gamma_{T\setminus R}\phantom{\boxtimes \orient^k_R}.
\end{equation}
(This is correct even for $k=0$ if we consider $Y_\Gamma$ to have a single $(-1)$-simplex.) Therefore the cellular chain complex of $Y_\Gamma$ has
\[\widetilde{C}_{k-1}(Y_\Gamma)_T=\bigoplus_{\substack{R\subset T\\\abs{R}=k}}\Z[\Gamma_T/\Gamma_{T\setminus R}]\boxtimes \orient^k_R.\]
Here $\orient^k_R$ arises as the orientation (or fundamental class) of the simplices; it records the fact that $\FI$-morphisms can reverse the orientation of simplices.  Comparing with the explicit description of $X_\Gamma$ in \eqref{eq:XGammadef}, we find
\[((X_\Gamma)_T)_k=\bigoplus_{\substack{T=S\disjoint R\\\abs{R}=k}}\Z[\Gamma_T/\Gamma_S]\boxtimes \orient^k_R\]
from which the desired identification $\widetilde{C}_{k-1}(Y_\Gamma)\iso (X_\Gamma)_k$ is clear. (Note the shift in indexing, which is why $\widetilde{C}_\bullet(Y_\Gamma)$ is isomorphic to the shifted $X_\Gamma[1]$ rather than $X_\Gamma$.)
\end{proof}
\begin{remark}
If $\Gamma$ is not saturated, $(X_\Gamma)_T$ is still the shifted chain complex of a semi-simplicial set; what goes wrong is that e.g.\ multiple edges may have the same endpoints, so this fails to be a simplicial complex.
\end{remark}

\section{Congruence subgroups and complexes of split partial bases} \label{secSPB}
\subsection{The $\FI$-simplicial complex $Y_\Gamma$ for congruence $\FI$-groups}
\label{sec:describecomplex}
In this section, we give explicit and familiar descriptions of the $\FI$-simplicial complex $Y_\Gamma$ for the concrete $\FI$-groups $\GL(R)$ and its congruence subgroups $\GL(R,I)$. It turns out that for proper ideals $I$, this complex coincides with  a natural ``complex of mod-$I$ split partial bases'' $\SPB_n(R,I)$. (This is closely related to the complexes of split unimodular sequences considered by Charney~\cite{Charney} and Putman~\cite{Pu}, as we will explain below.) However there is a subtle point when the ideal $I$ is equal to $R$, in which case the group $\GL_n(R,I)$ is simply $\GL_n(R)$. In this case the complex $Y_\Gamma$ does \emph{not} coincide with the complex of split partial bases $\SPB_n(R)=\SPB_n(R,I)$, but is instead a \emph{slightly different} complex. Making this subtlety clear is one of the main reasons for writing this section.

Before dealing with ideals at all, we define the complex of split unimodular collections $\SU_n(R)$ and the complex of split partial bases $\SPB_n(R)$.

\begin{definition}
\label{def:SU}
The vertices of the complex of split unimodular collections $\SU_n(R)$ are pairs $(v\in R^n,g\colon R^n\onto R)$ with $g(v)=1$. A collection $\{(v_1,g_1),\ldots,(v_k,g_k)\}$ forms a $(k-1)$-simplex of $\SU_n(R)$ if and only if $v_1,\ldots,v_k$ are linearly independent and $g_i(v_j)=\delta_{ij}$.
\end{definition}
Note that the $(n-1)$-simplices of $\SU_n(R)$ are in bijection with unordered bases $\{v_1,\ldots,v_n\}$ of $R^n$ (since the maps $g_i$ are then determined by the formula $g_i(v_j)=\delta_{ij}$).

\begin{definition}
\label{def:SPB}
The complex of split partial bases $\SPB_n(R)$ is the subcomplex of $\SU_n(R)$ defined as follows. A simplex of $\SU_n(R)$ belongs to $\SPB_n(R)$ if and only if it is contained in an $(n-1)$-simplex.
\end{definition}

\begin{remark}
For $R=\Z$ the complexes $\SPB_n(\Z)$ and $\SU_n(\Z)$ are actually equal, but this is not true for all rings $R$. For example, if $R=\R[x,y,z]/(x^2+y^2+z^2-1)$, the vector $v=(x,y,z)\in R^3$ and $g\colon (a,b,c)\mapsto ax+by+cz$ define a vertex $(v,g)\in \SU_3(R)$. But this vertex cannot belong to $\SPB_3(R)$ because the kernel of $g$ is not free (a basis for $\ker g$ would define a trivialization of the tangent bundle of the 2-sphere).
\end{remark}

To define the mod-$I$ variants $\SU_n(R,I)$ and $\SPB_n(R,I)$ we must fix a standard basis $e_1,\ldots,e_n$ for $R^n$. Let $\lambda_1,\ldots,\lambda_n\colon R^n\to R$ be the dual basis.
\begin{definition}
The complex of mod-$I$ split unimodular collections  $\SU_n(R,I)$ is the full subcomplex of $\SU_n(R)$ on vertices $(v,g)$ for which there exists $i\in [n]$ such that $v\equiv e_i\bmod{I}$ and $g\equiv \lambda_i\bmod{I}$.
\end{definition}
Note that the $(n-1)$-simplices of $\SU_n(R,I)$ are in bijection with bases $\{v_1,\ldots,v_n\}$ of $R^n$ such that $v_i\equiv e_i\bmod{I}$ for all $i=1,\ldots,n$.

\begin{definition}
\label{def:SPB-modI}
The complex of mod-$I$ split partial bases $\SPB_n(R,I)$ is the subcomplex of $\SU_n(R,I)$ defined as follows. A simplex of $\SU_n(R,I)$ belongs to $\SPB_n(R,I)$ if and only if it is contained in an $(n-1)$-simplex of $\SU_n(R,I)$.
\end{definition}

\begin{remark}
\label{remark:transitivityrange}
If $R$ satisfies Bass's stable range condition $\SR_{d+2}$, then for all $\ell\leq n-d-2$ the subcomplex $\SPB_n(R,I)$ contains all $\ell$-simplices of $\SU_n(R,I)$ (see \cite[Lemma~3.2]{Pu} or \cite[Proposition on p.~2101]{Charney}, but note a typo in the former). In particular, the inclusion $\SPB_n(R,I)\into \SU_n(R,I)$ is $(n-d-2)$-connected.
\end{remark}
\begin{remark}
\label{remark:SPBtilde}
Putman works with $\SU_n(R,I)$ in \cite[third Definition in \S3]{Pu} (but beware an inaccurate reference to it as the complex of split partial bases).
Charney does not work directly with the simplicial complex $\SU_n(R,I)$, but rather with a semi-simplicial set $\SUtilde_n(R,I)$ whose simplices are \emph{ordered} sequences $((v_1,g_1),\ldots,(v_k,g_k))$. Nevertheless there is a surjection $\pi\colon \SUtilde_n(R,I)\to \SU_n(R,I)$ which admits a section (see \cite[Proof of Lemma 3.1]{Pu}), which will be enough for us.
\end{remark}

We can view these simplicial complexes as forming $\FI$-simplicial complexes as follows.
\begin{definition}
The $\FI$-simplicial complex $\SU(R)$ is defined as follows. For each $n$ we set $\SU_{[n]}(R)=\SU_n(R)$. An inclusion $f\colon [n]\into [m]$ induces maps $f_*\colon R^n\into R^m$ (defined by $f_*(e_i)=e_{f(i)}$) and $f^*\colon R^m\onto R^n$ (defined by $f^*(e_{f(i)})=e_i$ and $f^*(e_j)=0$ for $j\notin \im f$). We define the structure map $f_*\colon \SU_n(R)\to \SU_m(R)$ by sending a vertex $(v,g)\in \SU_n(R)$ to $(f_*(v),f^*\circ g)\in \SU_m(R)$. This preserves the condition defining simplices of $\SU_n(R)$, so extends to an injection of simplicial complexes.

These structure maps $f_*\colon \SU_n(R)\to \SU_m(R)$ preserve the subcomplexes $\SPB_n(R)$, $\SU_n(R,I)$, and $\SPB_n(R,I)$. Therefore we obtain $\FI$-simplicial complexes $\SPB(R)$, $\SU(R,I)$, and $\SPB(R,I)$.
\end{definition}

We can now describe the $\FI$-simplicial complex $Y_{\GL(R,I)}$ defined in Definition~\ref{def:YGamma}.
\begin{proposition}
\label{prop:YGammaSPB}
If $I\subset R$ is a \emph{proper} ideal, $Y_{\GL(R,I)}$ is isomorphic as a $\GL(R,I)$-equivariant $\FI$-simplicial complex to $\SPB(R,I)$.
\end{proposition}
\begin{proof}
Set $\Gamma=\GL(R,I)$, and recall that the vertices of $(Y_\Gamma)_n$ are pairs $(\{i\in [n]\},\proj{\gamma}\in \Gamma_n/\Gamma_{[n]\setminus\{i\}})$; for readability we write $(i,\proj{\gamma})$ in place of $(\{i\},\gamma)$.
We define the isomorphism $\varphi\colon Y_\Gamma\iso \SPB(R,I)$ on vertices by sending $(i,\proj{\id})\in (Y_\Gamma)_n$ to the standard vertices $x_i=(e_i,\lambda_i)\in \SPB_n(R,I)$. Since this isomorphism is to be $\GL(R,I)$-equivariant, we must take $\varphi(i,\proj{\gamma})=\gamma\cdot x_i$. Note that the stabilizer of $x_i$ is precisely $\Gamma_{[n]\setminus\{i\}}$, so this is well-defined. It also respects the $\FI$-structure maps on vertices.

We now verify that $\varphi$ is a simplicial map. The $(k-1)$-simplices of $(Y_\Gamma)_n$ are precisely those of the form $\gamma\cdot \sigma$ where $\sigma$ is a ``standard'' simplex $\{(i_1,\proj{\id}),\ldots,(i_k,\proj{\id})\}$. Note that $\sigma$ itself is taken to $\varphi(\sigma)=\{x_{i_1},\ldots,x_{i_k}\}$ which is certainly a simplex of $\SPB_n(R,I)$, since it belongs to the standard $(n-1)$-simplex $\{x_1,\ldots,x_n\}$. Since $\GL_n(R,I)$ acts on the simplicial complex $\SPB_n(R,I)$, we conclude that $\varphi(\gamma\cdot \sigma)=\gamma\cdot\varphi(\sigma)$ is a simplex as well. This shows that $\varphi\colon Y_\Gamma\to \SPB(R,I)$ is a map of $\FI$-simplicial complexes. It remains to check that it is an isomorphism.

For a map of simplicial complexes, injectivity can be checked on vertices, and surjectivity can be checked on maximal simplices. The fact that the stabilizer of $x_i$ is $\Gamma_{[n]\setminus \{i\}}$ shows that $\varphi(i,\proj{\gamma})=\varphi(i,\proj{\gamma'})\implies \proj{\gamma}=\proj{\gamma'}$. So to check injectivity on vertices, it suffices to check that we cannot have $\varphi(i,\proj{\gamma})=\varphi(j,\proj{\gamma'})$ when $i\neq j$. This is where we will use that $I$ is a \emph{proper} ideal. By definition $\varphi(i,\proj{\gamma})=\gamma\cdot (e_i,\lambda_i)=(\gamma\cdot e_i,\gamma^*\lambda_i)$, and since $\gamma\in \Gamma_n=\GL_n(R,I)$ we know that $\gamma\cdot e_i\equiv e_i\bmod{I}$. The key point is that since $I$ is a \emph{proper} ideal, $e_i\not\equiv e_j\bmod{I}$ for $i\neq j$. Therefore $\gamma\cdot e_i$ cannot coincide with $\gamma\cdot e_j$, verifying injectivity on vertices.

For surjectivity, note that the maximal simplices $\sigma$ of $\SPB_n(R,I)$ are $(n-1)$-dimensional, and correspond to bases $\{v_1,\ldots,v_n\}$ of $R^n$ such that for each $i$ there exists $j_i\in [n]$ such that $v_i\equiv e_{j_i}\bmod{I}$. Since $I$ is proper, after reordering we can guarantee that $v_i\equiv e_i\bmod{I}$. But then the vectors $v_i$ together define a matrix $\gamma\in \GL_n(R,I)$ such that $\gamma\cdot \{x_1,\ldots,x_n\}=\sigma$, and thus $\sigma=\varphi(\{(1,\proj{\id}), \ldots,(n,\proj{\id})\}$. This concludes the proof.
\end{proof}

\begin{remark}The assumption that $I$ is a proper ideal is really necessary here. If we attempt to carry out the same comparison when $I=R$, we encounter a discrepancy. Note that in this case $\GL(R,I)=\GL(R)$ and $\SPB(R,I)=\SPB(R)$.

The first key observation is that $\GL_n(R)$ acts transitively on the $\ell$-simplices of $\SPB_n(R)$ for all $\ell$. Indeed, since every simplex is contained in an $(n-1)$-simplex, and $\GL_n(R)$ acts transitively on these, it suffices to check that $\GL_n(R)$ acts transitively on the $\ell$-simplices contained in a single maximal simplex. But for this we need only the permutation matrices. (These would be excluded from $\GL_n(R,I)$ if $I$ were a proper ideal!)

On the other hand, we may take $\Gamma=\GL(R)$ and consider the simplicial complex $Y_n=(Y_\Gamma)_{[n]}$. But in contrast, $\Gamma_n$ cannot act transitively on $\ell$-simplices of $Y_n$; indeed by definition, the $\Gamma_n$-orbits of $(k-1)$-simplices of $Y_n$ are in bijection with $k$-element subsets of $[n]$. So $\SPB_n(R)$ definitely \emph{cannot} coincide with $Y_n$.

We can nevertheless describe $Y_n$: the simplices of $Y_n$ correspond to tuples $(S\subset [n], f\colon RS\into R^n, g\colon R^n\onto RS)$  such that $g\circ f =\id$. In particular, the vertices of $Y_n$ correspond to tuples $({i\in [n]}, f\colon R\{i\}\into R^n, g\colon R^n\onto R\{i\})$ such that $g\circ f = \id$. There is a natural projection $\pi\colon Y_n\onto \SPB_n(R)$ sending a vertex $(i,f,g)$ to $(v=f(1),g)\in \SPB_n(R)$, but it is not injective. A collection of $k$ vertices $\sigma=\{(i_j, f_j, g_j)\}$ forms a $(k-1)$-simplex of $Y_n$ if and only if \emph{the labels $i_1,\ldots,i_k$ are all distinct} and $\pi(\sigma)$ forms a $(k-1)$-simplex of $\SPB_n(R)$.
\end{remark}

\subsection{Bounding the $\FI$-homology of congruence chains}
\label{sec:hyperhomologyvanishescong}
We can now explain how to deduce the necessary bound on the $\FI$-homology of the chains on congruence subgroups from the work of Charney. The argument in this section below partly follows a portion of \cite[Proof of Proposition 5.13]{castelnuovo-regularity}, but both the connection with $\HH_k(C_\bullet\Gamma)$ and the identification with $\SPB(R,I)$ are new.
We begin by proving Theorem~\ref{athm:CGammaSPB}, which connects the $\FI$-homology of congruence subgroups with the equivariant homology of the complex of mod-$I$ split partial bases.
\begin{theoremCrestated}
Given a ring $R$ and a proper ideal $I\subset R$, and any coefficient group $A$, for all $k\geq 0$ we have  \[\HH_k(C_\bullet(\GL(R,I);A))\iso \widetilde{\rH}_{k-1}^{\GL(R,I)}(\SPB(R,I);A).\]
\end{theoremCrestated}
\begin{proof}[Proof of Theorem~\ref{athm:CGammaSPB}]
Let $\Gamma=\GL(R,I)$. Proposition~\ref{prop:homologyofchains} and Corollary~\ref{cor:homologyofchainscoeff} state that $\HH_k(C_\bullet(\Gamma;A))\iso \rH_k(\Gamma;X_\Gamma\otimes A)$.
Proposition~\ref{prop:XGammaYGamma} states that $X_\Gamma[1]$ is the reduced chain complex of $Y_\Gamma$, which is isomorphic to $\SPB(R,I)$ by  Proposition~\ref{prop:YGammaSPB}, so $\rH_k(\Gamma;X_\Gamma\otimes A)$ is the $\Gamma$-equivariant homology \[
\HH_k(C_\bullet(\Gamma;A))\iso \rH_k(\Gamma;X_\Gamma\otimes A)\iso \widetilde{\rH}_{k-1}^{\GL(R,I)}( \SPB(R,I); A).\qedhere\]
\end{proof}

We can now prove Proposition~\ref{hyperhomologyvanishescong}, which was the necessary technical input for Application~\ref{athm:congruence-subgroup}.
\begin{proof}[Proof of Proposition~\ref{hyperhomologyvanishescong}] For readability, we explain the argument with $\Z$ coefficients, but it applies verbatim with arbitrary coefficients.
Charney proved in \cite[Theorem 3.5]{Charney} that $\SUtilde_n(R,I)$ is $q$-acyclic for $n\geq 2q+d+3$.

Our first step is to verify that the same is true of $\SU_n(R,I)$ and $\SPB_n(R,I)$. By Remark~\ref{remark:SPBtilde}, the projection $\SUtilde_n(R,I)\onto \SU_n(R,I)$ has a section, so it is surjective on homology; thus $\SU_n(R,I)$ is similarly $q$-acyclic for $n\geq 2q+d+3$. By Remark~\ref{remark:transitivityrange}, the inclusion $\SPB_n(R,I)\into \SU_n(R,I)$ is $(n-d-2)$-connected. Since $n\geq 2q+d+3$ implies $n-d-2\geq 2q+1>q$, we conclude that 
\begin{equation}
\label{eq:SPBacyclic}
\text{$\SPB_n(R,I)$ is $q$-acyclic for $n\geq 2q+d+3$.}
\end{equation}
This means $\widetilde{\rH}_q(\SPB_n(R,I))=0$ for $n\geq 2q+d+3$; in other words, the $\FI$-module $\widetilde{\rH}_q(\SPB(R,I))$ has \[\deg \widetilde{\rH}_q(\SPB(R,I))\leq 2q+d+2.\]

Note that $\deg \rH_p(G;V)\leq \deg V$ for any $V$ (simply because $\rH_p(G;0)=0)$.
It follows immediately via the spectral sequence \[E^2_{pq}=\rH_p(\GL(R,I);\widetilde{\rH}_q(\SPB(R,I)))\implies \widetilde{\rH}^{\GL(R,I)}_{p+q}(\SPB(R,I))\] that $\deg \widetilde{\rH}_k^{\GL(R,I)}(\SPB(R,I))\leq 2k+d+2$. Applying Theorem~\ref{athm:CGammaSPB}, we conclude that
\begin{equation}
\label{eq:degHC}
\deg \HH_k(C_\bullet \GL(R,I))=\deg \widetilde{\rH}_{k-1}^{\GL(R,I)}(\SPB(R,I))\leq 2k+d,
\end{equation} as desired.
\end{proof}

We conclude this paper by proving Theorem~\ref{athm:connectivity}, whose statement we recall for convenience.
\begin{theoremDduplicate}
Given any $\ell>0$, for each $k>0$ we have \[\widetilde{\rH}_{k-1}(\SPB_{2k}(\Z/p^\ell,p);\F_p)\neq 0.\]
Given any number ring $\cO$ and any prime power $p^a>2$, for each $k>0$ we have
\[\text{\textrm{either }} \widetilde{\rH}_{k-1}(\SPB_{2k}(\cO,p^a);\F_p)\neq 0\qquad\text{or}\qquad\widetilde{\rH}_{k-1}(\SPB_{2k+1}(\cO,p^a);\F_p)\neq 0.\]
\end{theoremDduplicate}
We remark that the same nonvanishing results apply to $\SU_n(\Z/p^\ell,p)$ and $\SUtilde_n(\Z/p^\ell,p)$ when $k>1$, since the inclusion $\SPB_{2k}(\Z/p^\ell,p)\into \SU_{2k}(R,I)$ is $(2k-2)$-connected. Similarly, the same results apply to $\SU_n(\cO,p^a)$ and $\SUtilde_n(\cO,p^a)$ when $k>2$.\begin{proof}
We first check that all these complexes are $(k-2)$-acyclic. We noted in \eqref{eq:SPBacyclic} that $\SPB_n(R,I)$ is $q$-acyclic for $n\geq 2q+d+3$. All the rings $R$ occuring in the proposition have dimension 0 or 1, so $d\leq 1$. Since $2k+1\geq 2k\geq 2(k-2)+d+3$, Charney's results show that all these complexes are $(k-2)$-acyclic as claimed.

The structure of the proof is as follows. The theorem deals with two cases: Case~A, when $R=\Z/p^\ell$ and $I=pR$, and Case~B, when $R=\cO$ and $I=p^aR$ for  $p^a>2$. The details will be quite different in places, but the overall argument is the same, so we first outline the proof in general. Let $\Gamma=\GL(R,I)$ and \[V_k\coloneq \HH_k(C_\bullet (\Gamma;\F_p))\iso \widetilde{\rH}_{k-1}^{\Gamma}(\SPB(R,I);\F_p).\]  In both Case~A and B, we will show that if the theorem were false for a certain $k$, we could prove the upper bound \begin{equation}
\label{eq:degVkbound}
\deg V_k=\deg \widetilde{\rH}_{k-1}^{\Gamma}(\SPB(R,I);\F_p)\leq 2k-1.
\end{equation}
 (This argument is the first place  the two cases diverge.) By Theorem~\ref{thm:growth-bounds}\ref{part:deltaLEt} and Corollary~\ref{cor:homologyofchainscoeff}, we have \[\delta(\rH_k(\GL_n(R,I);\F_p))\leq \deg V_k.\]
Thanks to Proposition~\ref{prop:over-fields}, the bound $\delta(\rH_k(\GL_n(R,I);\F_p))\leq 2k-1$ would imply the upper bound $\dim \rH_k(\GL_n(R,I);\F_p)=O(n^{2k-1})$ as $n\to \infty$. We will then derive a contradiction by showing in both cases that known results imply 
\begin{equation}
\label{eq:dimHkTheta}
\dim \rH^k(\GL_n(R,I);\F_p)=\Theta(n^{2k})
\end{equation} for all $k$. (This is the second place  the two cases diverge.) To complete the proof, we must prove in both Case~A and Case~B that \eqref{eq:degVkbound} holds if the theorem is false, and prove \eqref{eq:dimHkTheta} in both Case~A and Case~B.

Proving \eqref{eq:degVkbound} in Case~A: Since $R=\Z/p^\ell$ has Krull dimension 0, we have $d=0$. Thus from \eqref{eq:degHC} we know that $\deg V_k\leq 2k$, so let us consider this $\FI$-module in degree $2k$.
Since the complex $\SPB_{2k}(R,I)$ is $(k-2)$-acyclic, we have an isomorphism
\[\rH_0(\Gamma_{2k};\widetilde{\rH}_{k-1}(\SPB_{2k}(R,I);\F_p))\iso \widetilde{\rH}^{\Gamma_{2k}}_{k-1}(\SPB_{2k}(R,I);\F_p)=(V_k)_{2k}.\]
In particular, we have a surjection 
\[\widetilde{\rH}_{k-1}(\SPB_{2k}(R,I);\F_p)\onto (V_k)_{2k}.\] Therefore if the theorem were false and $\widetilde{\rH}_{k-1}(\SPB_{2k}(R,I);\F_p)=0$ for a certain $k$, we would have $(V_k)_{2k}=0$ and thus $\deg V_k\leq 2k-1$. This verifies \eqref{eq:degVkbound} in Case~A.

The proof of \eqref{eq:degVkbound} in Case~B is very similar, except that since $R=\cO$ has Krull dimension 1, we only know from \eqref{eq:degHC} that $\deg V_k\leq 2k+1$. Just as above we have surjections $\widetilde{\rH}_{k-1}(\SPB_{2k}(R,I);\F_p)\onto (V_k)_{2k}$ and $\widetilde{\rH}_{k-1}(\SPB_{2k+1}(R,I);\F_p)\onto (V_k)_{2k+1}$. Therefore if the theorem were false and both these homology groups vanished for a certain $k$, we would have $(V_k)_{2k}=(V_k)_{2k+1}=0$ and thus $\deg V_k\leq 2k-1$. This verifies \eqref{eq:degVkbound} in Case~B.

\bigskip
The remainder of the paper consists of the proof of \eqref{eq:dimHkTheta}.
First, let us consider the simplest case when $R=\Z/p^2$ and $I=pR$. In this case $\Gamma_n=\GL_n(\Z/p^2,p)$ is an elementary abelian group isomorphic to $(\Z/p)^{n^2}$, so the K\"unneth theorem implies that $\rH^*(\Gamma_n;\F_p)\iso \rH^*(\Z/p;\F_p)^{\otimes n^2}$. Since $\dim \rH^k(\Z/p;\F_p)=1$ for all $k$, we find that $\dim \rH^k(\Gamma_n;\F_p)$ is the coefficient of $t^k$ in $(1+t+t^2+\cdots)^{n^2}=\frac{1}{(1-t)^{n^2}}$, namely $\binom{n^2 + k-1}{k}$. In particular, this shows that $\dim \rH^k(\Gamma_n;\F_p)=\Theta(n^{2k})$. 

For $R=\Z/p^\ell$ in general, $\Gamma_n=\GL_n(\Z/p^\ell,p)$ is a non-abelian $p$-group. However, results on $p$-central groups imply that it nevertheless has the same cohomology as $(\Z/p)^{n^2}$, see Browder--Pakianathan~\cite[Corollary 2.34]{BrowderPakianathan}. Therefore as before we have $\dim \rH^k(\GL_n(\Z/p^\ell,p);\F_p)=\Theta(n^{2k})$. This finishes the proof in Case~A when $R=\Z/p^\ell$. (To extend the theorem  from $I=pR$ to $I=p^aR$ in this case, all that would be necessary is to show $\dim \rH^k(\GL_n(\Z/p^\ell,p^a);\F_p)\ \gg_n\ n^{2k-1}$. Such estimates may well already be known. Note that when $a\geq \ell/2$ this group is abelian, so this bound holds in that case.)

\medskip
We now turn to Case~B, when $R=\cO$ is a number ring. In this case for technical reasons we work with $\Gamma=\SL(\cO,p^a)$ rather than $\Gamma'=\GL(\cO,p^a)$. The complex $X_\Gamma$ agrees with $X_{\Gamma'}$ except in the very top dimensions, so all the bounds above work the same way. In particular, just as above, if the theorem were false we would have $\dim \rH_k(\SL_n(\cO,p^a);\F_p)=O(n^{2k-1})$. This contradicts the recent results of Calegari~\cite[Lemma 4.5, Remark 4.7]{Calegari}, which imply that this dimension is $\Theta(n^{2k})$. Note, however, that the statements there omit the hypothesis that $p^a > 2$. In order to be self-contained, and because the argument is short, we take this opportunity to summarize the argument of Calegari.

Let $\cO_p$ denote the $p$-adic completion of the number ring $\cO$. To address the cohomology of $\Gamma_n=\SL_n(\cO,p^a)$, we must first understand the continuous cohomology of the corresponding congruence group $G_n=\SL_n(\cO_p,p^a)$. Suppose the number ring $\cO$ has degree $D$. The pro-$p$ group $G_n$ is a compact $p$-adic analytic group of dimension $D(n^2-1)$, and our assumption that $p^a>2$ guarantees that it is torsion-free and uniformly $p$-powerful. The work of Lazard thus implies that $G_n$ is a Poincar\'e duality group of dimension $D(n^2-1)$ for continuous cohomology with $\F_p$ coefficients; in fact,
\begin{equation}
\label{eq:cohomology-p-adic}
\rH^*(G_n;\F_p)\iso \bwedge^* \rH^1(G_n;\F_p)\iso \bwedge^* ({\F_p}^{D(n^2-1)}).
\end{equation}
In particular, $\dim \rH^k(G_n;\F_p)=\binom{D(n^2-1)}{k}=\Theta(n^{2k})$. (Note that throughout, $\rH^*(G_n;M)$ denotes the \emph{continuous} cohomology of the profinite group $G_n$. If we knew $\dim \rH^k(G_n;\F_p)=\Theta(n^{2k})$ for the \emph{discrete} cohomology we could add the case $R=\cO_p$ to the theorem; alternately, the argument bounding $\delta(\rH_k(G;\F_p))$ could perhaps be modified to work with continuous cohomology.)
See \cite{SymondsWeigel} for a very readable overview of the cohomology of $p$-adic analytic groups; \eqref{eq:cohomology-p-adic} appears as \cite[Theorem 5.1.5]{SymondsWeigel}.

To connect this back to the arithmetic group, let $W^q_n$ denote the ``cohomology at infinite level'' \[W^q_n=\lim_{\substack{\longrightarrow\\r}}\rH^q(\SL_n(\cO,p^r);\F_p).\] Note that $W^q_n$ naturally inherits an action of $\lim_{\leftarrow}\SL_n(\cO/p^r)=\SL_n(\cO_p)$, which in fact extends to an action of $\SL_n(\cO_p\otimes \Q)$ via Hecke operators (though we will not really need this).  There is a Hochschild--Serre spectral sequence
\[E_2^{pq}=\rH^p(G_n; W^q_n)\quad\implies\quad \rH^{p+q}(\Gamma_n;\F_p)\]

The main result of Calegari--Emerton~\cite[Theorem~1.1]{CalegariEmerton} is that for $n\gg q$, the vector space $W^q_n$ is independent of $n$ and $\SL_n(\cO_p\otimes \Q)$ acts trivially on it (so in particular, so does $G_n$). This means that for sufficiently large $n$ we can (in a range of cohomological degrees) write this spectral sequence as 
\[E_2^{pq}=\rH^p(G_n;\F_p)\otimes W^q \quad\implies\quad \rH^{p+q}(\Gamma_n;\F_p)\]

The focus of Calegari's paper is the determination (as far as possible) of the stable cohomology groups $W^q$. But he points out that even without knowing anything, this spectral sequence allows us to estimate the dimension of $\rH^k(\Gamma_n;\F_p)$. From the computation of $\rH^k(G_n;\F_p)$ above, and the fact that $W^q=W^q_n$ does not depend on $n$, we find that $\dim E_2^{pq}=\Theta(n^{2p})$. In particular, for fixed $k$ the dimension of those $E_2^{pq}$ with $p+q=k$ is dominated by $E_2^{k0}=\rH^k(G_n;\F_p)$ whose dimension is $\Theta(n^{2k})$. All other terms in this string, as well as all those which could map to $E_r^{k0}$, have dimensions which are $O(n^{2k-2})$. Therefore without knowing anything about the behavior of this spectral sequence, we can conclude that $\dim \rH^k(\Gamma_n;\F_p)=\dim \rH^k(\SL_n(\cO,p^a);\F_p)=\Theta(n^{2k})$. This conclusion is \cite[Lemma~4.5]{Calegari} for $\cO=\Z$ and \cite[Remark~4.7]{Calegari} in general, except the hypothesis $p^a>2$ is missing from both (and beware a typo in the latter, where $N^{2kd}$ should be $d^kN^{2k}$). This completes the proof.
\end{proof}




\end{document}